\documentclass[12pt,a4paper,psamsfonts]{amsart}

\usepackage{fancyhdr}
\usepackage{appendix}
\usepackage{amssymb,amscd,amsxtra,calc}
\usepackage{mathrsfs}
\usepackage{cmmib57}
\usepackage{multirow}
\usepackage[all]{xy}
\usepackage{longtable}
\usepackage[colorlinks,linkcolor=blue,anchorcolor=blue,citecolor=green]{hyperref}
\theoremstyle{plain}
    \newtheorem{thm}{Theorem}[section]
    
    \newtheorem{claim}[thm]{Claim}
     \newtheorem{conjecture}[thm]{Conjecture}
    \newtheorem{corollary}[thm]{Corollary}
    
    \newtheorem{lemma}[thm]{Lemma}
    \newtheorem{proposition}[thm]{Proposition}
    \newtheorem{question}[thm]{Question}
    \newtheorem{theorem}[thm]{Theorem}

\theoremstyle{definition}
    \newtheorem{definition}[thm]{Definition}
    
    \newtheorem*{notation*}{Notation and Terminology}
    \newtheorem{remark}[thm]{Remark}
\theoremstyle{remark}

\newcommand{\C}{\mathbb{C}}

\newcommand{\PP}{\mathbb{P}}

\newcommand{\Q}{\mathbb{Q}}

\newcommand{\R}{\mathbb{R}}

\newcommand{\Z}{\mathbb{Z}}

\newcommand{\EFF}{\operatorname{E}}

\newcommand{\id}{\operatorname{id}}

\newcommand{\Ker}{\operatorname{Ker}}
\newcommand{\NE}{\overline{\operatorname{NE}}}
\newcommand{\Nef}{\operatorname{Nef}}

\newcommand{\NS}{\operatorname{NS}}

\newcommand{\PE}{\operatorname{PE}}

\newcommand{\Per}{\operatorname{Per}}

\newcommand{\rank}{\operatorname{rank}}
\newcommand{\SEnd}{\operatorname{SEnd}}

\newcommand{\Supp}{\operatorname{Supp}}

\newcommand{\Codim}{\operatorname{codim}}

\newcommand{\N}{\operatorname{N}}

\newcommand{\Alb}{\operatorname{Alb}}

\newcommand{\Pic}{\operatorname{Pic}}

\begin{document}

\title[Kawaguchi-Silverman conjecture]
{Kawaguchi-Silverman conjecture for certain surjective endomorphisms}

\author{Sheng Meng, De-Qi Zhang}

\address
{
\textsc{School of Mathematical Sciences, Shanghai Key Laboratory of PMMP} \endgraf
\textsc{East China Normal University, 500 Dongchuan Road, Shanghai 200241, People's Republic of China}\endgraf
\textsc{Max-Planck-Institut f\"ur Mathematik, Vivatsgasse 7, Bonn 53111, Germany} \endgraf
}
\email{smeng@math.ecnu.edu.cn}

\address
{
\textsc{Department of Mathematics} \endgraf
\textsc{National University of Singapore,
Singapore 119076, Republic of Singapore
}}
\email{matzdq@nus.edu.sg}

\begin{abstract}
We prove the Kawaguchi-Silverman conjecture (KSC), about the equality of arithmetic degree and
dynamical degree, for every surjective endomorphism of any (possibly singular) projective surface.
In high dimensions, we show that KSC holds for {\it every} surjective endomorphism of any $\Q$-factorial Kawamata log terminal projective variety admitting one int-amplified endomorphism, provided that KSC holds for any surjective endomorphism with the ramification divisor being totally invariant and irreducible.
In particular, we show that KSC holds for {\it every} surjective endomorphism of any rationally connected smooth projective threefold admitting one int-amplified endomorphism.
The main ingredients are the equivariant minimal model program, the effectiveness of the anti-canonical divisor and a characterization of toric pairs.
\end{abstract}
\makeatletter
\@namedef{subjclassname@2020}{\textup{2020} Mathematics Subject Classification}
\makeatother


\maketitle
\tableofcontents

\section{Introduction}

We work over an algebraically closed field $k$ of characteristic zero.
Let $f: X \to X$ be a surjective endomorphism of a projective variety $X$ over
$\overline{\mathbb{Q}}$.
There are two fundamental dynamical invariants from the aspects of topology and arithmetic.
The {\it first dynamical degree} $\delta_f$ is defined as
$$\delta_f:=\lim\limits_{m\to+\infty}((f^m)^*H\cdot H^{\dim(X)-1})^{1/m},$$
where $H$ is an ample divisor of $X$.
Such limit, independent of the choice of $H$, exists and equals to the spectral radius of $f^*|_{\NS(X)\otimes_{\Z} \C}$; see Definition \ref{def-d1}.
The {\it arithmetic degree} is defined as a function
$$\alpha_f(x) := \lim\limits_{m\to+\infty} \max\{1, h_H (f^m(x))\}^{1/m},$$
where $h_H$ is a Weil height function associated with an ample divisor $H$ of $X$; see Definition \ref{def-a-deg}.
The {\it Kawaguchi - Silverman Conjecture} ({\it KSC} for short, see \cite{KS-R}) asserts that $\alpha_f$ is well-defined, i.e., the limit exists, and $\alpha_f(x)=\delta_f$ for any point $x$ with Zariski dense $f$-orbit.

\begin{conjecture}\label{conj-KSC} (Kawaguchi-Silverman Conjecture = KSC)
Let $f:X\to X$ be a surjective endomorphism of a projective variety $X$ over $\overline{\mathbb{Q}}$.
Then the following hold.
\begin{itemize}
\item[(1)] The limit defining arithmetic degree $\alpha_f(x)$ exists for any $x \in X(\overline{\mathbb{Q}})$.
\item[(2)] If the (forward) orbit $O_f(x) = \{f^n(x)\, |\, n\ge 0\}$ is Zariski dense
in $X$, then the arithmetic degree of $x$ is equal to the dynamical degree of $f$, i.e.,
$\alpha_f(x) = \delta_f$.
\end{itemize}
\end{conjecture}

\begin{remark}
The original conjecture is formulated for dominant rational self-maps of smooth projective varieties.
In our setting, Conjecture \ref{conj-KSC} (1) has been proved by Kawaguchi and Silverman themselves (cf.~\cite{KS-TAMS});
more precisely, $\alpha_f(x)$ is either $1$ or the absolute value of an eigenvalue of $f^*|_{\N^1(X)}$ for any $x\in X(\overline{\mathbb{Q}})$.
In particular, $\alpha_f(x)\le \delta_f$.

Conjecture \ref{conj-KSC} (2) has been proved
at least in the following cases.

\begin{itemize}
\item[(i)]
$f$ is polarized (\cite[Theorem 5]{KS-TAMS}).
\item[(ii)]
$X$ is a smooth projective surface and $f$ is an automorphism (\cite[Theorem 2(c)]{KS14}).
\item[(iii)]
$X$ is a smooth projective surface and $f$ is non-isomorphic (\cite[Theorem 1.3]{MSS}).
\item[(iv)]
$X$ is a Mori dream space
(eg.~of Fano type; see \cite[Theorem 4.1, Corollary 4.2]{Mat}).
\item[(v)]
$X$ is an abelian variety
(\cite[Corollary 32]{KS-TAMS}, \cite[Theorem 2]{Sil17}).
\item[(vi)] $X$ is a Hyperk\"ahler variety (\cite[Theorem 1.2]{LS}).
\item[(vii)] $X$ is a smooth projective 3-fold with 
$\kappa(X) = 0$ and $\deg f > 1$ (\cite[Prop 1.6]{LS}).
\end{itemize}
\end{remark}

Theorems \ref{main-thm-surface}, \ref{main-thm-kappa}, \ref{main-thm-tir} and \ref{main-thm-src3} below are our main results.

First, we look at a surjective endomorphism $f$ of a (possibly singular) surface
$X$.
By taking the normalization which is $f$-equivariant, we may assume $X$ is normal.
When $f$ is an automorphism, one can further take an $f$-equivariant resolution and reduce to the smooth case; see \cite[Theorem 2(c)]{KS14}.

Assume now that $f$ is non-isomorphic. Wahl \cite[Theorem 2.8]{Wah} showed that $X$ has at worst {\it log canonical} (lc) singularities.
In the smooth  surface case, KSC is confirmed by Matsuzawa, Sano and Shibata \cite[Theorem 1.3]{MSS}, by reducing the problem to three precise cases: $\mathbb{P}^1$-bundles, hyperelliptic surfaces, and surfaces of Kodaira dimension one.

However, in the singular surface case, it is in general not possible to find an $f$-equivariant resolution.
Nevertheless, we are able to run an $f$-equivariant minimal model program (MMP) after iterating $f$; see Section \ref{sec-sur-emmp}.
Our key observation is Theorem \ref{thm-surf-rho2} (see also Theorem \ref{thm-surf-nonpe}) which shows that the only troubled case of Fano contraction, involved in the  KSC, is in fact of product type; see also Theorem \ref{thm-tir-1} for a higher dimensional analogue.
Conjecture \ref{conj-KSC} is thus fully solved for surfaces in Theorem \ref{main-thm-surface}.

When $\deg (f) \ge 2$, our proof (for possibly singular surfaces) does not depend on (and in fact recover) \cite[Theorem 1.3]{MSS} (for the smooth surface case).

\begin{theorem}\label{main-thm-surface}
KSC holds for any surjective endomorphism of a projective surface.
\end{theorem}

We now look at a higher dimensional projective variety $X$.
A surjective endomorphism $\mathcal{P}:X\to X$ is said to be {\it $q$-polarized} if $\mathcal{P}^*H\sim qH$ for some ample (integral) Cartier divisor $H$ and $q>1$.
A surjective endomorphism $\mathcal{I}:X\to X$ is said to be {\it int-amplified} if $\mathcal{I}^*L-L=H$ for some ample Cartier divisors $L$ and $H$.
Every polarized endomorphism is int-amplified.
See \cite{Meng_IAMP}, \cite{MZ}, \cite{MZ_PG} and \cite{MZ_S} for properties of such $\mathcal{P}$ or $\mathcal{I}$.

Especially,
$\mathcal{I}:X\to X$ is int-amplified if and only if
every eigenvalue of
the pullback action $\mathcal{I}^*|_{\NS(X)_{\C}}$ on the complex N\'eron-Severi space is of modulus $>1$;
further,
for any given surjective endomorphisms $f, g$ of $X$, if $f$ is int-amplified
(or say polarized), then $f^N \circ g$ is also int-amplified for some large $N$ (cf.~\cite[Theorem 1.4]{Meng_IAMP}). Thus, our working condition of assuming the existence of one
int-amplified endomorphism in Question \ref{que-ksc-iamp} and Theorems \ref{main-thm-kappa}, \ref{main-thm-tir} and \ref{main-thm-src3} is quite flexible.

Let $f:X\to X$ be a (not necessarily int-amplified) surjective endomorphism.
We wish to run an MMP  $f$-equivariantly (after replacing $f$ by a positive power).
On the one hand, to run an MMP, we need to assume that $X$ has only mild singularities, eg. $\Q$-factorial {\it Kawamata log terminal} (klt) singularities; see \cite[Definition 2.34]{KM}, \cite{BCHM}.
On the other hand, for the $f$-equivariance, we need to assume that $X$ admits at least one int-amplified endomorphism; see \cite[Theorems 1.1 and 1.2]{MZ_PG}.

Therefore, in higher dimensions, we focus on the following question:

\begin{question}\label{que-ksc-iamp}
Let $X$ be a normal projective variety which has only $\Q$-factorial Kawamata log terminal (klt) singularities and admits one int-amplified endomorphism.
Does KSC hold for every surjective endomorphism of $X$?
\end{question}

In \cite[\S 5]{Mat}, Matsuzawa provided a possible solution by adding three more assumptions: the {\it anti-Kodaira dimension} $\kappa(X, -K_X)>0$, $X$ being rationally connected, and the flip termination conjecture.
The flip termination conjecture is proved when $\dim(X)\le 3$ (cf.~\cite{Mor}, \cite{Sho}). However, it remains very difficult in higher dimensions.
On the other hand, it is proved in \cite[Theorem 1.1]{CMZ} of authors' joint paper that $-K_X$ is numerically effective when $X$ admits a polarized endomorphism.
This result was further generalized by the first author to the int-amplified case \cite[Theorem 1.5]{Meng_IAMP}.
In general, a numerically effective divisor may not be effective.
Nevertheless, we are able to strengthen \cite[Theorem 1.5]{Meng_IAMP} and show below that $-K_X$ is indeed effective, or equivalently $\kappa(X,-K_X)\ge 0$.

\begin{theorem}[cf.~Theorem \ref{thm-kappa}]\label{main-thm-kappa}
	Let $X$ be a $\Q$-Goreinstein normal projective variety admitting one int-amplified endomorphism.
	Then we have:
\begin{itemize}
\item[(1)]
$-K_X\sim_{\Q} D$ ($\Q$-linear equivalence) for some effective $\Q$-Cartier divisor $D$.
\item[(2)]
Suppose further the anti-Kodaira dimension $\kappa(X, -K_X)= 0$.
Then $D$ is a reduced Weil divisor such that $g^{-1}(D)=D$ and $g|_{X \backslash D}: X \backslash D \to X \backslash D$
is quasi-\'etale, i.e., \'etale in codimension $1$,   for any surjective endomorphism $g$ of $X$.
\end{itemize}
\end{theorem}

In view of Theorem \ref{main-thm-kappa}, we are led to the case $\kappa(X,-K_X)=0$, or $\kappa(X,-K_X)>0$.
In \cite[Proposition 3.6]{Mat} (cf.~Proposition \ref{prop-ksc-kappa(D)>=0}), Matsuzawa showed that KSC holds for $f$ if $f^*D\sim \delta_f D$ with $\kappa(X,D)>0$.
In general, one cannot weaken the linear equivalence assumption here to numerical equivalence if $H^1(X,\mathcal{O}_X)\neq 0$.
However, we are able to prove the following, by using the {\it anti-Kodaira fibration} and the {\it Chow reduction}; see \S \ref{sec-kappa>0}.

\begin{proposition}\label{main-thm-caseB}
Let $f:X\to X$ be a surjective endomorphism of a $\Q$-Goreinstein normal projective variety $X$ with the anti-Kodaira dimension $\kappa(X,-K_X)>0$.
Suppose $f^*K_X\equiv \delta_f K_X$ (numerical equivalence).
Then KSC holds for $f$.
\end{proposition}

Let $\pi:X\to Y$ be a finite surjective endomorphism of normal projective varieties.
Denote by $R_{\pi}$ the ramification divisor of $\pi$ so that $K_X = \pi^*K_Y + R_{\pi}$.
For a surjective endomorphism $f:X\to X$, it is said to have {\it totally invariant ramification} if
$f^{-1}(\Supp R_f)=\Supp R_f$.

By another key observation (cf.~Proposition \ref{lem-fano}) and induction on the dimension, after running an equivariant minimal model program, we may assume $f^*K_X\equiv \delta_f K_X$, or else KSC holds.
We then further show that we are only left with the following case for Question \ref{que-ksc-iamp} by virtue of Proposition \ref{main-thm-caseB}.
Here, we remark in advance that Condition (A5) below is implied by Conditions (A1) - (A4); see Theorem \ref{thm-tir-1}.

\par \vskip 1pc
{\bf Case TIR}$_n$ (Totally Invariant Ramification case).
Let $X$ be a normal projective variety of dimension $n \ge 1$, which has only $\Q$-factorial Kawamata log terminal (klt) singularities and admits one int-amplified endomorphism.
Let $f:X\to X$ be an arbitrary surjective endomorphism.
Moreover, we impose the following conditions.
\begin{itemize}
\item[(A1)]
The anti-Kodaira dimension $\kappa(X,-K_X)=0$; $-K_X$ is nef, whose class is extremal in both the {\it nef cone} $\Nef(X)$ and the {\it pseudo-effective divisors cone} $\PE^1(X)$.
\item[(A2)]
$f^*D = \delta_f D$ for some prime divisor $D\sim_{\Q} -K_X$.
\item[(A3)]
The ramification divisor of $f$ satisfies $\Supp \, R_f = D$.
\item[(A4)]
There is an $f$-equivariant Fano contraction $\pi:X\to Y$ with $\delta_f>\delta_{f|_Y}$ ($\ge 1$).
\item[(A5)]
$\dim(X) \ge \dim(Y)+2 \ge 3$.
\end{itemize}

Our following result is also a bit surprise to us, where we are able to reduce Question \ref{que-ksc-iamp} to the above highly geometrically restrictive Case TIR, by merely starting from an arbitrary endomorphism without any pre-given geometric information.

\begin{theorem}\label{main-thm-tir}
Let $X$ be a normal projective variety having only $\Q$-factorial Kawamata log terminal (klt)  singularities and one int-amplified endomorphism.
Then we have:
\begin{itemize}
\item[(1)] If $K_X$ is pseudo-effective, then KSC holds for any surjective endomorphism of $X$.
\item[(2)] Suppose that KSC holds for Case TIR
(for those $f|_{X_i} : X_i \to X_i$ appearing in any equivariant MMP starting from $X$).
Then KSC holds for any (not necessarily int-amplified) surjective endomorphism $f$ of $X$.
\end{itemize}
\end{theorem}

As a consequence, Question \ref{que-ksc-iamp} can be reduced to the following:
\begin{question}\label{que-tir}
	Does there exist $f: X \to X$ satisfying Case TIR (plus, if necessary, that $X$ is rationally connected as defined below)?
	If such $f$ exists,
	does it satisfy KSC?
\end{question}

\begin{remark}
Condition (A5) of Case TIR implies that $\dim(X) \ge 3$.
Recently, Matsuzawa and Yoshikawa constructed in \cite[\S 7]{MY} an interesting example: a klt rational surface $X$ satisfying all the conditions of Case TIR$_2$ except (A4) and (A5).
Moreover, $X$ admits an (equivariant) quasi-\'etale cover which is a (smooth) ruled surface over an elliptic curve,
and the totally invariant divisor $D$ there is an elliptic curve.
\end{remark}

A projective variety $X$ is said to be {\it rationally connected}, in the sense of Campana and Kollar-Miyaoka-Mori
(\cite{Ca}, \cite{KMM}),
if two general points of $X(\C)$ are connected by a rational curve, after taking one (and hence every) embedding of the defining field of $X$ into $\C$; see also \cite[Definition 3.2, Exercise 3.2.5]{Kol}.

Let $X$ be a rationally connected smooth projective variety admitting an int-amplified endomorphism $f$ with totally invariant ramification.
In \cite[Corollary 1.4]{MZ_toric}, the authors showed that $X$ is then toric if $f$ is polarized.
For the int-amplified case, the difficulty lies in showing the semistablity for the reflexive sheaf of germs of logarithmic 1-forms; see Section \ref{sec-toric} for the details.
Nevertheless, we are able to prove the following:

\begin{proposition}\label{main-prop-toric-iamp} (cf.~Proposition \ref{prop-toric-iamp})
Let $f:X\to X$ be an int-amplified endomorphism of a rationally connected smooth projective variety $X$ with totally invariant ramification,
i.e., $f^{-1}(\Supp R_f) = \Supp R_f$.
Suppose that $X$ admits some MMP
$$X=X_1\dashrightarrow\cdots \dashrightarrow X_r\to Y=\mathbb{P}^1$$
where $X_i\dashrightarrow X_{i+1}$ is birational and $X_r\to Y$ is a Fano contraction.
Then $X_i$ is a toric variety for each $i$.
In particular, KSC holds for any surjective endomorphism of $X_i$.
\end{proposition}

By Proposition \ref{main-prop-toric-iamp}, one can rule out Case TIR$_3$ during any MMP starting from a rationally connected smooth projective threefold.
Namely, we have:

\begin{theorem}\label{main-thm-src3}
Let $X$ be a rationally connected smooth projective threefold admitting one int-amplified endomorphism.
Then KSC holds for an arbitrary surjective endomorphism of $X$.
\end{theorem}

\begin{remark} [{\bf Applications of our results}]
Our approach in this paper lays down the frame work to attack the KSC for an arbitrary surjective endomorphism on a (possibly singular or non-rationally connected) projective
variety $X$ admitting an int-amplified endomorphism, by reducing to Case TIR$_n$; see Questions \ref{que-ksc-iamp} and \ref{que-tir}.
For instance, the steps, including the equivariant MMP, and the results of the current paper are essentially used/restated by Matsuzawa and Yoshikawa \cite[Proposition 5.1, Lemmas 5.2-5.3, Remark 1.2]{MY_rc}.
Indeed, furthering our exclusion of Case TIR$_n$ when $n = 3$,
they have applied our equivariant MMP and made progress towards our frame work by excluding Case TIR during the MMP when $X$ is smooth and rationally connected. It shows that to attack the KSC for arbitrary (possibly singular) variety it is important/inevitable to study Case TIR$_n$ to which narrowed down by the current paper.
\end{remark}

\par \vskip 1pc
{\bf Acknowledgement.}
Many thanks to Y.~Matsuzawa for inspiring discussions, to F.~Hu, and the referee for valuable suggestions to improve this paper, to Max Planck Institute for Mathematics, Bonn for the first author's Postdoc Fellowship, to organisers of Simons Symposium on Algebraic, Complex and Arithmetic Dynamics, May 2019, for the opportunity of a talk.
The first author is supported in part by Science and Technology Commission of Shanghai Municipality (No. 22DZ2229014).
The second author is supported by an ARF of NUS.

\section{Preliminaries}
\begin{notation*}\label{notation}
Let $X$ and $Y$ be projective varieties of dimension $n$.
Let $f:X\to X$ be a surjective endomorphism  and $\pi:X\to Y$ a finite surjective morphism.
We say $\pi$ is {\it quasi-\'etale} if it is \'etale in codimension $1$.

Two $\R$-Cartier divisors $D_i$ of $X$ are {\it numerically equivalent}, denote as $D_1\equiv D_2$, if $(D_1 - D_2)\cdot C=0$ for any curve $C$ on $X$.
Two $r$-cycles $C_i$ of $X$ are {\it weakly numerically equivalent}, denoted as $C_1 \equiv_w C_2$, if
$(C_1 - C_2)\cdot L_1 \cdots L_{n - r} = 0$ for
all Cartier divisors $L_i$.
The numerical equivalence implies weak numerical equivalence; see \cite[Section 2]{MZ}.

We use the following notation throughout the paper unless otherwise stated.
\renewcommand*{\arraystretch}{1.1}

\begin{longtable}{p{2cm} p{12cm}}
$\Pic(X)$    &the group of Cartier divisors of $X$ modulo linear equivalence $\sim$\\
$\Pic_{\mathbb{K}}(X)$    &$\Pic(X)\otimes_{\Z}\mathbb{K}$ with $\mathbb{K}=\mathbb{Q}, \mathbb{R},\mathbb{C}$\\
$\Pic^0(X)$    &the neutral connected component of $\Pic(X)$\\
$\Pic^0_{\mathbb{K}}(X)$    &$\Pic^0(X)\otimes_{\Z}\mathbb{K}$ with $\mathbb{K}=\mathbb{Q}, \mathbb{R},\mathbb{C}$\\
$\NS(X)$    &$\Pic(X)/\Pic^0(X)$, the N\'eron-Severi group\\
$\N^1(X)$    &$\NS(X)\otimes_{\Z}\mathbb{R}$, the space of $\R$-Cartier divisors modulo numerical equivalence $\equiv$\\
$\NS_{\mathbb{K}}(X)$    &$\NS(X)\otimes_{\Z}\mathbb{K}$ with $\mathbb{K}=\mathbb{Q}, \mathbb{R},\mathbb{C}$\\
$\N_r(X)$    &the space of $r$-cycles modulo weak numerical equivalence $\equiv_w$\\
$f^*|_V$    & the pullback action on $V$, which is any group or space above\\
$f_*|_V$    & the pushforward action on $V$, which is any group or space above\\
$\Nef(X)$   & the cone of nef classes in $\N^1(X)$\\
$\NE(X)$   & the cone of pseudo-effective classes in $\N_1(X)$\\
$\PE^1(X)$  & the cone of pseudo-effective classes in $\N^1(X)$\\
$R_{\pi}$    & the ramification divisor of $\pi$ assuming that $X$ and $Y$ are normal\\
$\Supp D$   & the support of $D=\sum a_i D_i$ which is $\bigcup_i D_i$, where $a_i>0$ and $D_i$ are prime divisors\\
$\SEnd(X)$ &the monoid of all the surjective endomorphisms of $X$\\
$\kappa(X,D)$ & Iitaka dimension of a $\Q$-Cartier divisor $D$\\
$\rho(X)$    & Picard number of $X$ which is $\dim_{\R} \N^1(X)$
\end{longtable}

\end{notation*}

\begin{definition}
Let $f:X\to X$ be a surjective endomorphism of a variety $X$ and $Z\subseteq X$ a subset.
$Z$ is said to be {\it $f$-invariant} (resp. {\it $f^{-1}$-invariant}) if $f(Z)=Z$ (resp.~$f^{-1}(Z)=Z$).
$Z$ is said to be {\it $f$-periodic} (resp. {\it $f^{-1}$-periodic}) if $f^s(Z)=Z$ (resp.~$f^{-s}(Z)=Z$) for some $s>0$.
\end{definition}

\begin{definition}\label{def-d1}(Dynamical degree;
$\delta_f$, $\iota_f$)
Let $f:X\to X$ be a surjective endomorphism of a projective variety $X$.
The (first) {\it dynamical degree} $\delta_f$ of $f$ is defined as the spectral radius of $f^*|_{\N^1(X)}$.
Another equivalent definition is
$$\delta_f=\lim\limits_{m\to+\infty}((f^m)^*H\cdot H^{\dim(X)-1})^{1/m},$$
where $H$ is any nef and big Cartier divisor of $X$.
Denote by $\iota_f$ the {\it minimal modulus of eigenvalues} of $f^*|_{\N^1(X)}$.
When $X$ is smooth over the complex field, $\delta_f$ (resp.~$\iota_f$) is equal to the maximal (resp.~minimal) modulus of eigenvalues of $f^*|_{H^{1,1}(X,\R)}$ (cf.~\cite{DS04}, \cite[\S4]{DS17}).
Note that $\delta_{f^s}=(\delta_f)^s$.
\end{definition}

\begin{definition}\label{def-a-deg}(Weil height function and arithmetic degree)
Let $X$ be a normal projective variety defined over $\overline{\mathbb{Q}}$.
We refer to \cite[Part B]{HS}, \cite{KS-R} or \cite[Section 2.2]{Mat} for the detailed definition of the {\it Weil height function} $h_D:X(\overline{\mathbb{Q}})\to \R$ associated with some $\R$-Cartier divisor $D$ on $X$.
Here, we simply list some fundamental facts which will be used later.
\begin{itemize}
\item $h_D$ is determined only up to a bounded function by the divisor $D$.
\item $h_{\sum a_i D_i}=\sum a_i h_{D_i}+O(1)$ where $O(1)$ means some bounded function.
\item $h_E$ is bounded below outside $\Supp E$ for any effective Cartier divisor $E$.
\item Let $\pi:X\to Y$ be a surjective morphism of normal projective varieties and $B$ some $\R$-Cartier divisor of $Y$.
Then $h_B(\pi(x))=h_{\pi^*B}(x)+O(1)$ for any $x\in X(\overline{\mathbb{Q}})$.
\end{itemize}

The {\it arithmetic degree} $\alpha_f(x)$ of $f$ at $x \in X(\overline{\mathbb{Q}})$ is defined as
$$\alpha_f(x) = \lim\limits_{m\to+\infty} \max\{1, h_H (f^m(x))\}^{1/m},$$
where $H$ is an ample Cartier divisor.
This limit exists and is independent of the choice of $H$ (cf.~\cite[Theorem 2]{KS-TAMS}, \cite[Proposition 12]{KS-R}).
Moreover, $\alpha_f(x)$ is either $1$ or the absolute value of an eigenvalue of $f^*|_{\N^1(X)}$ (cf.~\cite[Remark 23]{KS-TAMS}).
Note that $\alpha_f(x)\le \delta_f$ and $\alpha_{f^s}(x)=\alpha_f(x)^s$.
This allows us to replace $f$ by any positive power whenever needed.
\end{definition}

In the rest of this section, we list several fundamental results about KSC which are important and will be frequently used in the rest of the paper.

\begin{lemma}\label{lem-d1-bir}
Let $\pi:X\dashrightarrow Y$ be a dominant rational map of projective varieties.
Let $f:X\to X$ and $g:Y\to Y$ be surjective endomorphisms such that $g\circ \pi=\pi\circ f$.
Then $\delta_g\le \delta_f$.
Further, if $\pi$ is generically finite, then $\delta_g= \delta_f$.
\end{lemma}

\begin{proof}
For the convenience of the reader, we give a quick proof of this well known result.
Let $W$ be the graph of $\pi$ and $p_X:W\to X$ and $p_Y:W\to Y$ the two projections.
Here $p_X$ is a birational morphism and $p_Y$ is a surjective morphism.
Denote by $h:W\to W$ the lifting of $f$.
Let $H$ be any ample Cartier divisor of $X$.
By the projection formula, $$\delta_f = \lim\limits_{m\to+\infty} [(f^m)^*H\cdot H^{\dim(X)-1}]^{1/m}=\lim\limits_{m\to+\infty} [(h^m)^*(p_X^*H)\cdot (p_X^*H)^{\dim(W)-1}]^{1/m}=\delta_h$$ since $p_X^*H$ is nef and big.
Note that $p_Y^*:\N^1(Y)\to \N^1(W)$ is injective.
So $\delta_g\le \delta_f$.
Suppose $\pi$ is generically finite.
Let $A$ be an ample divisor of $Y$.
Then $p_Y^*A$ is nef and big.
A similar argument shows that $\delta_g=\delta_f$.
\end{proof}

The proof of the following lemma is taken from \cite[Lemma 5.6]{Mat}.
\begin{lemma}\label{lem-ksc-iff}
Let $\pi:X\dashrightarrow Y$ be a dominant rational map of projective varieties.
Let $f:X\to X$ and $g:Y\to Y$ be surjective endomorphisms such that $g\circ \pi=\pi\circ f$.
Then the following hold.
\begin{itemize}
\item[(1)] Suppose $\pi$ is generically finite.
Then KSC holds for $f$ if and only if KSC holds for $g$.
\item[(2)] Suppose $\delta_f=\delta_g$ and KSC holds for $g$.
Then KSC holds for $f$.
\end{itemize}
\end{lemma}

\begin{proof}
For (1), by taking the graph of $\pi$, it suffices for us to consider the case when $\pi$ is a generically finite surjective morphism.
By Lemma \ref{lem-d1-bir}, $\delta_f=\delta_g$.
Let $x$ be a closed point of $X$.
It is clear that $\overline{O_f(x)}=X$ if and only if $\overline{O_g(\pi(x))}=Y$.
Take any $x\in X$ with Zariski dense orbit.
Let $H$ be an ample Cartier divisor of $Y$.
We have
$$h_H(g^m(\pi(x))) = h_H(\pi(f^m(x))) = h_{\pi^*H}(f^m(x))+ O(1).$$
So $\alpha_g(\pi(x))\le \alpha_f(x)$.
Since $\pi$ is generically finite, we may write $\pi^*H=A+E$ for some ample Cartier divisor $H$ and effective Cartier divisor $E$ after replacing $H$ by a multiple.
There exists an infinite sequence $n_1<n_2<\cdots$ such that $\{f^{n_i}(x)\,|\, i=1, 2, \cdots\}$ is Zariski dense in $X$ and $f^{n_i}(x)\not\in \Supp E$.
Since $h_E$ is bounded below outside $\Supp E$,
we have  $$h_H(g^{n_i}(\pi(x)))=h_A(f^{n_i}(x))+h_E(f^{n_i}(x))+O(1)\ge h_A(f^{n_i}(x))+O(1).$$
This implies that $\alpha_g(\pi(x))\ge \alpha_f(x)$.
So (1) is proved.

For (2), we may assume that $\pi$ is a surjective morphism by (1).
By the first equality, we have $\delta_g= \alpha_g(\pi(x))\le \alpha_f(x)\le\delta_f$ and (2) is proved.
\end{proof}

\begin{lemma}\label{lem-ksc-x}(cf.~\cite[Lemma 3.2]{San})
Let $f:X\to X$ and $g:Y\to Y$ be two surjective endomorphisms of projective varieties.
Suppose KSC holds for both $f$ and $g$.
Then KSC holds for $f\times g$.
\end{lemma}

\begin{proposition}\label{prop-ksc-kappa(D)>=0}(cf.~\cite[Proposition 3.6]{Mat})
Let $f:X\to X$ be a surjective endomorphism of a normal projective variety $X$.
Suppose $f^*D\sim_{\Q}\delta_f D$ for some effective $\Q$-Cartier divisor with Iitaka dimension $\kappa(X,D)>0$.
Then KSC holds for $f$.
\end{proposition}

\begin{theorem}\label{thm-ksc-qa}(cf.~\cite[Theorem 2]{Sil17})
Let $X$ be a $Q$-abelian variety, i.e., it has a quasi \'etale cover by an abelian variety.
Then KSC holds for any surjective endomorphism of $X$.
\end{theorem}

\begin{proof}
Let $g:X\to X$ be a surjective endomorphism.
There exists a finite surjective morphism $\pi:A\to X$ with $A$ being an abelian variety, such that $g$ lifts to a surjective endomorphism $f:A\to A$
(cf.~\cite{Na-Zh} or \cite[Corollary 8.2]{CMZ}).
Then the result follows from \cite[Theorem 2]{Sil17} and Lemma \ref{lem-ksc-iff}.
\end{proof}

\section{Pullback Action on $\Pic(X)$}
In this section, we discuss the relation between $f^*|_{\Pic^0(X)}$ and $f^*|_{\N^1(X)}$.
\begin{proposition}\label{prop-eig-sqrt-abe}
	Let $f:A\to A$ be an isogeny of an abelian variety $A$.
	Denote by $A_{\C}:=A\otimes_{\mathbb{Z}}\C$ and $f_{\C}:A_{\C}\to A_{\C}$ the induced linear map.
	Let $\lambda$ be an eigenvalue of $f_{\C}$.
	Then $\iota_f\le |\lambda|^2\le \delta_f$
	(cf. Definition \ref{def-d1}).
\end{proposition}

\begin{proof}
	After embedding the defining field of $A$ and $f$ in $\C$, we may assume that $A$ is defined over $\C$.
	Suppose $f_{\C}(x)=\lambda x$ for some $\lambda\neq 0$ and $0\neq x\in A_{\C}$.
	Let $P_f\in \mathbb{Z}[t]$ be the characteristic polynomial of $f^*|_{H^1(A,\mathbb{Z})}$.
	Then $P_f(f)=0$ and hence $P_f(f_{\C})=(P_f(f))_{\C}=0$.
	In particular, $P_f(f_{\C})(x)=0$.
	Then $\lambda$ is a root of $P_f$ and hence an eigenvalue of $f^*|_{H^1(A,\mathbb{Z})}$.
	Therefore, $|\lambda|^2$ is an eigenvalue of $f^*|_{H^{1,1}(A,\mathbb{R})}$.
The proposition is proved.
\end{proof}

\begin{lemma}\label{lem-eig-dual}
	Let $f:A\to A$ be a surjective endomorphism of an abelian variety $A$.
	Let $f^{\vee}:A^{\vee}\to A^{\vee}$ be the dual endomorphism of the dual abelian variety $A^{\vee}:=\Pic^0(A)$.
	Then $\delta_f=\delta_{f^{\vee}}$ and $\iota_f=\iota_{f^{\vee}}$.
\end{lemma}
\begin{proof}
We may replace the base field by $\C$.
Note that the dual of a translation is still a translation and the pullback action of a translation on $\N^1(A)$ is always an identity.
So we may assume that $f$ is an isogeny.
Let $m_f\in \mathbb{Z}[t]$ be the minimal polynomial of $f^*|_{H^1(A,\mathbb{Z})}$.
Then $m_f(f)=0$ and $m_f(f^{\vee})=m_f(f)^{\vee}=0$.
A dual argument shows that $m_f$ is also the minimal polynomial of $(f^{\vee})^*|_{H^1(A^{\vee},\mathbb{Z})}$.
Therefore, $f^*|_{H^{1,1}(A,\R)}$ and $(f^{\vee})^*|_{H^{1,1}(A^{\vee},\R)}$ have the same eigenvalues.
The lemma is proved.
\end{proof}

\begin{proposition}\label{prop-eig-sqrt}
	Let $f:X\to X$ be a surjective endomorphism of a normal projective variety $X$ whose Albanese morphism is surjective.
	Let $\lambda$ be an eigenvalue of $f^*|_{\Pic^0_{\C}(X)}$.
	Then $\iota_f\le |\lambda|^2\le \delta_f$.
\end{proposition}
\begin{proof}
	Let $\pi:X\to A$ be the Albanese morphism.
	Note that $A$ is the dual of $\Pic^0(X)$.
	Denote by $g:=f|_A$.
	Then $g^{\vee}=f^*|_{\Pic^0(X)}+a$ for some $a\in \Pic^0(X)$.
	Since $\pi$ is surjective,
	we have $\iota_f\le \iota_g=\iota_{g^{\vee}-a}\le\delta_{g^{\vee}-a}=\delta_g\le \delta_f$ by Lemma \ref{lem-eig-dual}.
Then the result follows from Proposition \ref{prop-eig-sqrt-abe}.
\end{proof}

\begin{proposition}\label{prop-iamp->1}
	Let $f:X\to X$ be an int-amplified endomorphism of a normal projective variety $X$.
	Then all the eigenvalues of $f^*|_{\Pic_{\Q}(X)}$ are of modulus greater than $1$.
\end{proposition}
\begin{proof}
Note that $\NS_{\C}(X)=\Pic_{\C}(X)/\Pic^0_{\C}(X)$ and all the eigenvalues of $f^*|_{\NS_{\C}(X)}$ are of modulus greater than $1$ by \cite[Theorem 1.1]{Meng_IAMP}.
	By Proposition \ref{prop-eig-sqrt}, all the eigenvalues of $f^*|_{\Pic^0_{\C}(X)}$ are of modulus greater than $1$.
The result follows.
\end{proof}

\begin{lemma}\label{lem-pol-pic}(cf.~\cite[Lemma 19]{KS-TAMS})
	Let $f:X\to X$ be a morphism.
	Then there is a monic integral polynomial $P_f(t) \in \mathbb{Z}[t]$ with the property that $P_f(f^*)$ annihilates $\Pic(X)$.
\end{lemma}

\begin{definition}\label{def-vf}
Let $f:X\to X$ be a surjective endomorphism of a projective variety $X$.
Let $D\in \Pic_{\R}(X)$.
Denote by $V_f(D)$ the subspace of $\Pic_{\R}(X)$ spanned by $\{(f^m)^*D\}_{m\ge 0}$.
Denote by $\EFF_{f}(D)$ the convex cone of effective $\R$-Cartier divisors in $V_{f}(D)$.
Note that $\EFF_{f}(D)$ does not contain any line.
However, the closure of $\EFF_f(D)$ may contain lines.
\end{definition}

We need the following to show the effectiveness of anti-canonical divisor in Section \ref{sec-eff}.

\begin{proposition}\label{prop-vf}
Let $f:X\to X$ be a surjective endomorphism of a projective variety $X$.
Then the following hold.
\begin{itemize}
\item[(1)] For any $D\in \Pic_{\R}(X)$, $V_f(D)$ and $\EFF_{f}(D)$ are finite dimensional and $f^*|_{\Pic_{\R}(X)}$-invariant.
\item[(2)] $f_*f^*=f^*f_*=(\deg f)\id$ on $\Pic_{\R}(X)$.
\end{itemize}
\end{proposition}
\begin{proof}
    By Lemma \ref{lem-pol-pic}, $V_f(D)$ is finite dimensional.
    Clearly, $f^*(V_f(D))\subseteq V_f(D)$.
    By the projection formula, $f_*f^*=(\deg f)\id$ on $\Pic(X)$.
    So $f^*|_{\Pic_{\R}(X)}$ is injective, hence $f_*f^*=f^*f_*$ on $V_f(D)$.
    Note that $f^*D$ is effective if $D$ is effective.
    So $f^*(E_f(D))=E_f(D)$.
\end{proof}

\section{Equivariant Minimal Model Program for Surfaces}\label{sec-sur-emmp}
In this section, we recall the (monoid) equivariant minimal model program for a (possibly singular) normal projective surface admitting a non-isomorphic endomorphism.

\begin{lemma}\label{lem-intersection}
	Let $X$ be a normal projective surface and $C$ an irreducible curve on $X$.
	Then there exists an integer $n_0 >0$ (depending only on $X$) such that $n_0 C\equiv_w D$ (weak numerical equivalence) for some (integral) Cartier divisor $D$.
\end{lemma}

\begin{proof}
Let $D_1,\cdots, D_r$ be (integral) Cartier divisors which form a basis of $\N^1(X)$.
Denote by $A:=(D_i\cdot D_j)_{1\le i,j\le r}$ the intersection matrix which is invertible by the Hodge index theorem.
Then there is some $D=\sum a_i D_i$ such that $D\cdot D_i=C\cdot D_i\in \Z$ for each $i$.
Since
$(a_1,\cdots,a_r)\in A^{-1}(\mathbb{Z}^r)\subseteq \mathbb{Z}^r/\det(A)$,
we are done by letting $n_0 =\det(A)$.
\end{proof}

Let $X$ be a normal projective surface. By \cite[Lemma 3.2]{Zh-tams}, there is a natural embedding $\N^1(X) \subseteq \N_1(X)$.
Let $C$ and $C'$ be two curves on $X$.
Then we may define $C\cdot C':=D\cdot C'$ for some $D\in \N^1(X)$ with $D\equiv_w C$, which is independent of the choice of $D$ (cf.~Lemma \ref{lem-intersection}).
We say that $C$ has {\it negative self-intersection} if $C^2<0$.
Note that the above intersection coincides with the traditional Mumford intersection for Weil divisors on $X$.
Indeed, given a resolution $\pi:X'\to X$, the Mumford pullback preserves the weak numerical equivalence, i.e., $\pi^*C\equiv_w \pi^*D$, by noting that $\pi^*C\cdot E=0$ for any exceptional prime divisor $E$.
Denote by $R_C := \mathbb{R}_{\ge 0}[C]$ the {\it ray generated by} $[C]$ in $\NE(X)$. Denote by $\Sigma_C$ the {\it union of curves} whose classes are in $R_C$.
Let $f:X\to X$ be a surjective endomorphism.
The projection formula implies that $f(\Sigma_C)=\Sigma_{f(C)}$ and $f^{-1}(\Sigma_C)=\Sigma_{C'}$ for any curve $C'$ with $f(C')=C$; see \cite[Lemma 4.2]{MZ_PG}.

\begin{lemma}\label{lem-surf-exc}
Let $X$ be a normal projective surface with only log canonical (lc) singularities.
Let $\pi:X\to Y$ be a divisorial contraction of some $K_X$-negative extremal ray having the exceptional divisor $E=\sum E_i$ with $E_i$ irreducible.
Then $Y$ has only lc singularities. Further, $E_i^2<0$ and $\Sigma_{E_i}=E$.
\end{lemma}
\begin{proof}
$Y$ is lc by \cite[Theorem 3.3]{Fuj12}.
In particular, $K_Y$ is $\Q$-Cartier.

Write $K_X \sim_{\Q} \pi^*K_Y+\sum a_i E_i$.
Since $K_X\cdot E_i<0$, we have $a_i> 0$ for each $i$ by the negativity lemma (cf.~\cite[Lemma 3.39]{KM}).
Note that the rays $R_{E_i}=R_{E_j}$ in $\N_1(X)$ and $\Sigma_{E_i}=E$.
Then $E_1\equiv_w t(\sum a_i E_i)$ for some $t>0$.
Since $(\sum a_i E_i)\cdot E_1=K_X\cdot E_1<0$, we have $E_1^2<0$.
\end{proof}

Let $X$ be a normal projective surface.
Denote by $S(X)$ the set of {\it all irreducible curves} $C$ on $X$ with negative self-intersection and $\Sigma_C$ being a finite union of irreducible curves.

\begin{lemma}\label{lem-sx}
(cf.~\cite{Nak02})
Suppose $X$ is a normal projective surface. Then we have:
	\begin{itemize}
		\item[(1)] The action $\SEnd(X)$ on $S(X)$, via $(f,C)\mapsto f(C)$, is well defined.
		\item[(2)] Suppose $X$ has a non-isomorphic surjective endomorphism. Then  $S(X)$ is finite;
		and $f^{-t_0}(C) = C$ for any $f\in \SEnd(X)$ and $C \in S(X)$
		where $t_0 = |S(X)|!$.
	\end{itemize}
\end{lemma}

\begin{proof}
For (1), let $f\in \SEnd(X)$, $C\in S(X)$.
	By Lemma \ref{lem-intersection}, $n_0 C\equiv_w D$ for some fixed integer $n_0>0$ and (integral) Cartier divisor $D$.
	Write $f_*C=df(C)$.
	Then $f(C)\equiv_w f_*D/(dn)$.
	By the projection formula,
	$$f(C)^2=(f_*D)^2/(d^2n^2)=(\deg f/d^2)C^2<0.$$
	On the other hand, $\Sigma_{f(C)}=f(\Sigma_C)$.
	Therefore (1) is proved.
	
	For (2), let $g\in \SEnd(X)$ be a non-isomorphic one.
	\begin{claim}\label{cl-s-1}
	For any $f\in \SEnd(X)$ and $C\in S(X)$, $f^{-1}f(C)=C$.
	\end{claim}
	Since $f^{-1}f(\Sigma_C)=\Sigma_C$,
	our $f^{-1}$ induces a bijection between the (finitely many) irreducible components of $\Sigma_{f(C)}$ and $\Sigma_C$.
	Since $C\subseteq f^{-1}f(C)$, the claim is proved.
	
		\begin{claim} \label{cl-s-2}
		For some $t>0$,
		$g^t(C)\subseteq \Supp R_g$,  where $R_g$ is the ramification divisor of $g$.
	\end{claim}
	
	Suppose the contrary.
	By Claim \ref{cl-s-1}, we have
	$g^*(g^{t}(C))=g^{t-1}(C)$ and hence
	$(g^t)_*C=(\deg g)^t g^t(C)$ for any $t>0$.
	Therefore,
	$$g^{t}(C)^2=(\deg g)^{-t}C^2<0.$$
	By Lemma \ref{lem-intersection}, $n_0^2g^t(C)^2\in\mathbb{Z}_{<0}$ for any $t>0$.
	Note that $C^2<0$ and $\deg g>1$.
	Then we get a contradiction by letting $t\gg 1$.
The claim is proved.
	
	Denote by $$S_0(X):=\{C\in S(X)\,|\, C\subseteq \Supp R_g\}$$ which is a finite set.
	For any $C\in S(X)$, $g^i(C)=g^j(C)\in S_0(X)$ for some $i>j>0$ by Claim \ref{cl-s-2}.
	Let $s_C=i-j$ which is determined by $C$.
	Then $$C=g^{-i}g^{i}(C)=g^{-i}g^j(C)=g^{-s_C}(C).$$
	\begin{claim} \label{cl-s-3}
Let $s=\prod_{C\in S_0(X)} s_C$. Then $S(X)=\bigcup_{i=0}^{s-1} g^i(S_0(X))$, hence it is a finite set.
	\end{claim}
	Let $C\in S(X)$.
	By Claim \ref{cl-s-2}, $C_0:=g^t(C)\in S_0(X)$ for some $t>0$.
	There exist some integers $a>0$ and $b\ge 0$ such that $as=t+b$ and $0\le b<s$.
	By Claim \ref{cl-s-1} and the choice of $s$, we have $C=g^{-t}g^t(C)= g^{-t} (C_0) = g^{-t}g^{as}(C_0)=g^b(C_0)$.
The claim is proved.

Finally, by (1) and Claim \ref{cl-s-3}, for any $f\in \SEnd(X)$ and $C\in S(X)$,
we have $f^i(C)=f^j(C)$ for some $i>j>0$ with $i-j\le |S(X)|$.
By Claim \ref{cl-s-1}, $C=f^{-i}f^i(C)=f^{-(i-j)}(C)$.
So (2) is proved.
\end{proof}

A submonoid $G$ of a monoid $\Gamma$ is said to be of {\it finite-index} in $\Gamma$
if there is a chain $G = G_0 \le G_1 \le \cdots \le G_r = \Gamma$ of submonoids and homomorphisms $\rho_i : G_i \to F_i$ such that $\Ker(\rho_i) = G_{i-1}$ and all $F_i$ are finite
{\it groups}.

\begin{theorem}\label{thm-emmp-sur}
	Let $X$ be a normal projective surface admitting a non-isomorphic surjective endomorphism. Then any MMP starting from $X$ is $G$-equivariant for some finite-index submonoid $G$ of $\SEnd(X)$.
\end{theorem}

\begin{proof}
By \cite[Theorem 2.8]{Wah}, $X$ has only lc singularities, so one can run MMP within the lc category (cf.~\cite[Theorem 1.1]{Fuj12}).
Any MMP of $X$ has at most $\rho(X)$ steps and involves only divisorial and Fano contractions.
Let $\pi:X\to Y$ be the first step.
Suppose $\pi$ is a Fano contraction.
By the finiteness of Fano contractions (cf.~\cite[Lemma 4.4]{MZ_PG}, \cite[Lemma 6.2]{MZ}), there is a submonoid $G\le \SEnd(X)$ such that $\pi$ is $G$-equivariant.
	
Suppose $\pi$ is divisorial.
By Lemma \ref{lem-surf-exc}, each irreducible component of the $\pi$-exceptional divisor is in $S(X)$.
By Lemma \ref{lem-sx}, $S(X)$ is finite and there is a submonoid $G\le \SEnd(X)$ of finite index such that $G|_{S(X)}=\id$.
So $\pi$ is $G$-equivariant.
Since $G$ and hence $G|_Y$ admit non-isomorphic endomorphisms,
we may replace $X$ by $Y$ and repeat the argument.
\end{proof}

\section{KSC for Surfaces: Proof of Theorem \ref{main-thm-surface}}

In this section, we will prove KSC for surfaces.
Indeed, we provide a very detailed characterization of a non-isomorphic surjective endomorphism $f:X\to X$ of a normal projective surface $X$.
Note that such $X$ has log canonical (lc) singularities by \cite[Theorem 2.8]{Wah}.
In particular, the canonical divisor $K_X$ is $\Q$-Cartier.

First, we recall a result of Nakayama which characterizes the case when the canonical divisor is pseudo-effective.

\begin{theorem}\label{thm-surf-pe}(cf.~\cite[Theorem 7.1.1]{ENS} or \cite[Theorem 3.1]{Nak-1934})
Let $f:X\to X$ be a non-isomorphic surjective endomorphism of a normal projective surface $X$ with $K_X$ being pseudo-effective.
Then
$K_X$ is nef,
$f$ is quasi-\'etale, and there is a quasi-\'etale finite Galois cover $\nu:V\to X$ such that $\nu\circ f_V=f^{\ell}\circ \nu$ for a non-isomorphic surjective endomorphism $f_V$ of $V$ and a positive integer $\ell$, and that $V$ and $\nu$ satisfy exactly one of the following conditions:
\begin{itemize}
\item[(1)] $V$ is an abelian surface.
\item[(2)] $V\cong E \times T$ for an elliptic curve $E$ and a smooth projective curve $T$ of genus $\ge 2$.
Moreover, $f_V$ and $f$ have no Zariski-dense orbit.
\end{itemize}
\end{theorem}

\begin{proof}
This follows from \cite[Theorem 7.1.1]{ENS} or \cite[Theorem 3.1]{Nak-1934} by letting the totally invariant divisor $S = 0$ there.
In fact, we only have Cases (3) and (2) there corresponding to our Cases (1) and (2) here.
For our Case (2), we only need to check the assertion about the non-existence of dense orbits.
For this, note that $f_V(E \times \{t\})$ has genus $\le 1$ (an elliptic curve, indeed)
and it cannot dominate $T$ which is of genus $\ge 2$.
Thus $f: V \to V$ descends to a surjective endomorphism $h: T \to T$ by the rigidity lemma \cite[Lemma 1.15]{Deb}.
Since $T$ has genus $\ge 2$, this $h$ has finite order.
So $f_V$ and hence $f$ have no Zariski-dense orbit.
\end{proof}	

We refer to \cite[Theorem 1.1 (4) iii]{Fuj11} for the cone theorem frequently used late on.

\begin{theorem}\label{thm-surf-rho2}
Let $X$ be a normal projective surface
with only log canonical singularities and
$\pi:X\to Y$ a Fano contraction with $\dim(Y)=1$.
Let $f:X\to X$ and $g:Y\to Y$ be surjective endomorphisms such that $g\circ \pi=\pi\circ f$.
Suppose $\delta_f>\delta_g$.
Then we have:
\begin{itemize}
\item[(1)] $f^*D\sim \delta_f D$ for some semi-ample and $\pi$-ample prime divisor $D$ with $R_D$ being an extremal ray of $\NE(X)$.
\item[(2)] There is a $\delta_f$-polarized endomorphism $h:\mathbb{P}^1\to \mathbb{P}^1$ such that $h\circ \phi=\phi\circ f$ where $\phi:X\to \mathbb{P}^1$ is the Iitaka fibration of $D$.
\end{itemize}
In particular, there is a finite surjective morphism $\tau:X\to \mathbb{P}^1\times Y$ such that $(g\times h)\circ \tau=\tau\circ f$.
\end{theorem}

\begin{proof}
First, $X$ has rational singularities, hence $\Q$-factorial (cf.~\cite[Lemmas 2.4.9 and 2.4.10]{ENS} or \cite[Proposition 2.33]{Nak-1825}).
By the assumption, the Picard number $\rho(X) = \rho(Y) + 1 = 2$.

Note that $\delta_g$ is a positive integer.
Since $\pi^*(\N^1(Y))$ is an $f^*$-invariant hyperplane of $\N^1(X)$, another eigenvalue $\delta_f$ of $f^*|_{\N^1(X)}$ is also an integer.
Let $F\cong \mathbb{P}^1$ be a general fibre of $\pi$.
Then $f^*F\equiv \delta_gF$.
Let $R_D$ be another extremal ray of $\Nef(X)$.
Then $D\cdot F>0$, and
$f^*D\equiv \delta_f D$.
We have $D^2 = 0$, for otherwise, $D^2>0$ and $$(\delta_f\delta_g)D^2=(\deg f)D^2=(f^*D)^2=(\delta_f)^2D^2$$ imply that $\delta_f=\delta_g$, contradicting the assumption.
Thus, $$\Nef(X) = \PE^1(X) = \NE(X).$$

\begin{claim}\label{c1}
Some choice of $D$ is $\Q$-Cartier and has $\kappa(X,D)>0$.
\end{claim}

Once Claim \ref{c1} is proved, the new $D$ has $\kappa(X,D)=1$ since $D^2=0$.
Then we have $D\sim_{\Q} D_a+D_c\sim_{\Q}D_b+D_c$
for effective $\Q$-Cartier divisors $D_a, D_b, D_c$ such that $\Supp D_a$ and $\Supp D_b$ are non-empty and have no common irreducible component.
Since $R_D$ is extremal in $\NE(X)$, we have $R_D=R_{D_a}=R_{D_b}$.
Then $D_a$ is nef and $D_a\cdot D_b=0$.
So $\Supp D_a \cap \Supp D_b=\emptyset$. In particular, $D_a$ is semi-ample.
Replacing $D$ by $mD_a$ for some $m>0$, we may assume $D$ is base point free.
Then the Iitaka fibration $\phi:X\to B$ is a morphism with $B$ being a smooth projective curve.
Note that $D \sim_{\Q} \phi^*H$ for some ample $\Q$-Cartier divisor $H$ of $B$.
Let $C$ be any irreducible curve of $X$.
Then $\phi(C)$ is a point if and only if $D\cdot C=H\cdot \phi_*C=0$.
Note that $f_*C\cdot D=C\cdot f^*D=\delta_f (C\cdot D)$.
So $\phi(C)$ is a point if and only if so is $\phi(f(C))$.
Since the Iitaka fibration $\phi$ has connected fibres, there is a surjective endomorphism $h:B\to B$ such that $h\circ \phi=\phi\circ f$ by the rigidity lemma (cf.~\cite[Lemma 1.15]{Deb}).
Note that $F$ dominates $B$ since $F\cdot D>0$.
Then $B\cong \mathbb{P}^1$ and $h^*|_{\Pic(B)}=\delta_f \id$.
In particular, $f^*D\sim_{\Q} f^*\phi^*H=\phi^*h^*H\sim_{\Q} \delta_f D$.

This proves the assertion (1) and (2) of the theorem.
For the final assertion of the theorem, $\tau$ is naturally induced by the two fibrations $\pi$ and $\phi$.
It is finite because $\rho(X)=\rho(\mathbb{P}^1\times Y)=2$.

Therefore, to prove the theorem, we only need to show Claim \ref{c1} which will be proved in several steps below.

{\bf Step 1.} Suppose $K_X\cdot D<0$.
By the cone theorem, $R_D$ is generated by a rational curve again denoted as $D$.
Note that $(aD-K_X)\cdot D>0$ and $(aD-K_X)\cdot F>0$ for $a>0$.
Then $aD-K_X$ is ample by Kleiman's ampleness criterion (cf.~\cite[Theorem 1.8]{KM}) and hence $D$ is semi-ample by the base point free theorem (cf.~\cite[Theorem 2.1]{Fuj11}).
So Claim \ref{c1} is proved in this case.

{\bf Step 2.} From now on, we assume that $K_X\cdot D\ge 0$.
Note that $$0\le D\cdot R_f=D\cdot (K_X-f^*K_X)=D\cdot K_X-f^*D\cdot f^*K_X/\delta_f=(1-\delta_g)D\cdot K_X\le 0.$$
Then $D\cdot R_f=0$.  Hence either $R_f=0$ or $R_D=R_{R_f}$.
If $R_f=0$, then $K_X=f^*K_X$ implies that $K_X$ is an eigenvector of  $\N^1(X)$.
So $K_X$ is numerically parallel to one of $D$ and $F$ and it must be the former since $-K_X$ and $D$ are relatively ample (but not $F$) over $Y$. Hence $1=\delta_f > \delta_g \ge 1$, a contradiction.
Therefore, $R_f\neq 0$ and $R_D=R_{R_f}$.
Write $R_f=\sum a_i D_i$ where $a_i>0$ are integers and $D_i$ are irreducible components.
Since $R_{R_f}$ is extremal in $\NE(X)$, $R_{D_i}=R_{D}$ for every $i$.

{\bf Step 3.}
Suppose $D_1$ is not $f^{-1}$-periodic.
Then there exists infinitely many different irreducible curves $E_t$ such that $f_*E_t=e_t E_{t-1}$ for some integer $e_t>0$ and $E_1=D_1$.
By Proposition \ref{prop-vf}, $E_t\sim_{\Q}(e_t/\deg f)f^*E_{t-1}$.
Then $V_f(D_1)$ (cf.~Definition \ref{def-vf}) is spanned by $\{E_t\}_{t\ge 0}$.
By Proposition \ref{prop-vf}, $V_f(D_1)$ is finite dimensional.
Then we have $A:=\sum_{i\in I} b_iE_i\sim_{\Q} \sum_{j\in J} b_jE_j=:B$ for some finite sets $I$ and $J$ with $I\cap J=\emptyset$ and $b_i,b_j$ are positive integers.
Note that $R_D=R_A$ and $\kappa(X, A)>0$.
So in this case, Claim \ref{c1} holds by taking $A$ as new $D$.

{\bf Step 4.}
Now we may assume that $f^{-1}(D_i)=D_i$ for every $i$ after replacing $f$ by a positive power.
Then $f^*D_i=\delta_f D_i$.

Suppose $\Supp R_f$ is not irreducible. Then we have $D_1\equiv tD_2$ for some rational number $t>0$.
Note that $m(D_1-tD_2)\in \pi^*(\Pic^0(Y))$ for some positive integer $m$  and $$f^*(D_1-tD_2)=\delta_f(D_1-tD_2).$$
If $D_1-tD_2 \sim_{\Q} 0$, then
$\kappa(X,D_1)>0$, and we are done.
Otherwise, $m(D_1-tD_2)\in \pi^*(\Pic^0(Y))$ is not a torsion. Hence $g^*$ has an eigenvector in $\Pic^0_{\C}(Y)$
corresponding to the eigenvalue $\delta_f>1$; thus the condition of Proposition \ref{prop-eig-sqrt}
cannot be satisfied, i.e., the Albanese morphism of $X$ is not surjective.
So the genus of $Y$ is at least $2$,
and then
$g$ has finite order and all the eigenvalues of $g^*|_{\Pic^0_{\C}(Y)}$ are roots of unity, again a contradiction.

{\bf Step 5.}
Finally, we are left with the case that $\Supp R_f=D_1$ is irreducible and $f^{-1}$-invariant.
Replace $D$ by $D_1$.
Then $K_X+D=f^*(K_X+D)$.
Note that
$$(K_X+D)\cdot F=f^*(K_X+D)\cdot f^*F/\delta_g=\delta_f(K_X+D)\cdot F.$$
So $(K_X+D)\cdot F=0$ and $D\cdot F=-K_X\cdot F=2$.

Let $\widetilde{X}$ be the normalization of the (irreducible) main component of $X\times_Y \widetilde{Y}$ where $\widetilde{Y}$ is the normalization of $D$.
Denote by $p_1:\widetilde{X}\to X$ and $p_2:\widetilde{X}\to \widetilde{Y}$ the induced projections.
Denote by $\widetilde{f}:\widetilde{X}\to \widetilde{X}$ the equivariant lifting of $f$ and $\widetilde{D}:=p_1^{-1}(D)$.
Note that there is a diagonal embedding $D\to D\times_Y D$ and $\pi|_D: D\to Y$ is a double cover.
Then $\widetilde{D}$ is not irreducible.
Note that the general fibre of $p_2$ is a smooth rational curve.
So $\widetilde{X}$ has only rational singularities and is $\Q$-factorial (cf.~\cite[Lemmas 2.4.9 and 2.4.10]{ENS} or \cite[Proposition 2.33]{Nak-1825}).
Write $\widetilde{D}:=\sum_{i=1}^2 \widetilde{D}_i$.
Replacing $\widetilde{f}$ by a positive power, we may assume $\widetilde{f}^{-1}(\widetilde{D}_i)=\widetilde{D}_i$ for each $i$.
Then $\widetilde{f}^*\widetilde{D}_i=\delta_f \widetilde{D}_i$.

We assert that $p_2: \widetilde{X}\to \widetilde{Y}$ is a Fano contraction.
First, $\deg (\widetilde{f}) = \deg(f) \ge 2$
implies that $\widetilde{X}$ is lc, thus we can run MMP of $X$ (cf.~\cite[Theorem 1.1]{Fuj12}).
Now let $\widetilde{C}$ be any negative curve of $\widetilde{X}$.
By Lemma \ref{lem-sx}, $\widetilde{f}^{-1}(\widetilde{C})=\widetilde{C}$ after replacing $\widetilde{f}$ by a positive power.
Write $\widetilde{f}_*(\widetilde{C})=t\widetilde{C}$ for some $t>0$.
Then $\widetilde{f}^*(\widetilde{C})=(\deg \widetilde{f} /t) \widetilde{C}$.
Now $f_*{p_1}_*\widetilde{C}={p_1}_*\widetilde{f}_*\widetilde{C}=t{p_1}_*\widetilde{C}$.
Since $p_1$ is finite, $p_1(\widetilde{C})$ is not a point and hence either $t=\delta_f$ or $\delta_g$.
However, $\widetilde{C}^2<0$ implies that $t^2 = \deg \widetilde{f}= \deg f$. Then $\delta_f=\delta_g$, a contradiction.
Thus the relative MMP of $\widetilde{X}$ over $\widetilde{Y}$ has only one step Fano contraction which is $p_2$ (as asserted).

Note that $\widetilde{D}\subseteq \Supp R_{\widetilde{f}}$.
By the same argument of Step 4, since $\widetilde{D}$ is not irreducible, we have $\kappa(\widetilde{X}, \widetilde{D}_i)>0$ and hence $\kappa(X,D)>0$.

So Claim \ref{c1} is proved in this case.
This also proves the theorem.
\end{proof}

We now characterize the case when the canonical divisor is not pseudo-effective.

\begin{theorem}\label{thm-surf-nonpe}
Let $f:X\to X$ be a non-isomorphic surjective endomorphism of a normal projective surface $X$ with $K_X$ not being pseudo-effective.
Then, replacing $f$ by a positive power, one of the following holds.
\begin{itemize}
\item[(1)] $f$ is polarized and $f^*|_{\N^1(X)}=q\id$ for some integer $q>1$.
\item[(2)] $\rho(X)=2$; there is an $f$-equivariant Fano contraction $\pi:X\to Y$ with $\delta_f=\delta_{f|_Y}$.
\item[(3)] $\rho(X)=2$; there exist a finite surjective morphism $\tau:X\to \mathbb{P}^1\times Y$ where $Y$ is a smooth projective curve, a surjective endomorphism $g:Y\to Y$, and a surjective endomorphism $h:\mathbb{P}^1\to \mathbb{P}^1$ such that $(g\times h)\circ \tau=\tau\circ f$.
\end{itemize}
\end{theorem}

\begin{proof}
Note that $X$ is lc by \cite[Theorem 2.8]{Wah}.
By \cite[Theorem 1.1]{Fuj12} and Theorem \ref{thm-emmp-sur}, replacing $f$ by a positive power, we may run $f$-equivariant MMP
	$$X = X_1 \to \cdots \to X_i\to \cdots \to X_r\to Y$$
with $\pi_i:X_i\to X_{i+1}$ being divisorial contractions for $i<r$ and $\pi_r:X_r\to Y$ being a Fano contraction.
Denote by $f_i:=f|_{X_i}$ and $g:=f|_Y$.
If $Y$ is a point, then $\rho(X_r)=1$ and $f_r$ is automatically polarized since $\deg f_r=\deg f>1$.
Note that if $f_r$ is polarized, then $f$ is polarized by \cite[Corollary 3.12]{MZ} and further $f^*|_{\N^1(X)}$ is a scalar action (cf.~\cite[Theorem 1.8]{MZ}).

Suppose now that $Y$ is a curve and $f_r$ is not polarized.
We claim that $r=1$.
Replacing $X$ by $X_{r-1}$, it suffices for us to consider the case when $r=2$.
Let $E$ be the exceptional divisor of $\pi_1:X\to X_2$.
Then $f^{-1}(E)=E$ and write $f^*E=tE$ for some $t>0$.
Let $P:=\pi_2\circ\pi_1(E)$ be a point in $Y$.
Then $g^*P=\delta_g P$.
Let $F_2:=\pi_2^*P$ and $F:=\pi_1^*F_2$.
Then $F=\widetilde{F_2}+aE$ where $a>0$ and $\widetilde{F_2}$ is the strict transform of $F_2$ in $X$.
Since $f^{-1}(\Supp F)=\Supp F$, we have $f^{-1}(\Supp \widetilde{F_2})=\Supp \widetilde{F_2}$.
Note that
$$\delta_g \widetilde{F_2}+ \delta_g aE=\delta_g F=f^*F=f^*\widetilde{F_2}+atE.$$
Therefore, $t=\delta_g$.
On the other hand, $E^2<0$ implies that $\delta_g^2=t^2=\deg f=\deg f_2$, hence the two eigenvalues of $f_2^*|_{\N^1(X_2)}$ are both $\delta_g$.
Since $\deg f>1$, $f$ is then polarized, a contradiction.
So the claim is proved.
In particular, $\rho(X)=2$.

The theorem is finished then by applying Theorem \ref{thm-surf-rho2}.
\end{proof}

\begin{remark}
In \cite[\S 7]{MY}, Matsuzawa and Yoshikawa constructed a family of int-amplified surjective endomorphisms $f:X\to X$ of a klt rational surface satisfying Theorem \ref{thm-surf-nonpe} (2) but not the others.
Their example has the properties: $\kappa(X,-K_X)=0$, and $-K_X\sim_{\Q}D$ with $D=\Supp R_{f}$ being an $f^{-1}$-invariant elliptic curve.
\end{remark}

\begin{proof}[Proof of Theorem \ref{main-thm-surface}]
We may assume that $X$ is normal after normalization by Lemma \ref{lem-ksc-iff}.
If $f$ is an automorphism, then we may further take an $f$-equivariant resolution and KSC holds for $f$ by \cite[Theorem 2(c)]{KS14} and Lemma \ref{lem-ksc-iff}.

Suppose $f$ is non-isomorphic.
Then $X$ is lc by \cite[Theorem 2.8]{Wah}.
If $K_X$ is pseudo-effective, then the theorem follows from Theorem \ref{thm-surf-pe}, \cite[Theorem 2]{Sil17} and Lemma \ref{lem-ksc-iff}.
If $K_X$ is not pseudo-effective, then the theorem follows from Proposition \ref{prop-ksc-kappa(D)>=0}, Theorem \ref{thm-surf-nonpe} and Lemmas \ref{lem-ksc-iff} and \ref{lem-ksc-x}.
\end{proof}

\section{Effectiveness of $-K_X$: Proof of Theorem \ref{main-thm-kappa}}\label{sec-eff}

In this section, we show the effectiveness of the anti-canonical divisor of any variety admitting an int-amplified endomorphism. Theorem \ref{thm-kappa} below includes Theorem \ref{main-thm-kappa}.

\begin{proposition}\label{prop-kap0-inv}
Let $f:X\to X$ be a surjective endomorphism of a normal projective variety $X$.
Let $D$ be an effective Cartier divisor of $X$ with $\kappa(X, D)=0$ and $f^*D\sim_{\Q}D+B$ for some effective $\Q$-Cartier divisor $B$.
Then $f^{-1}(\Supp D)=\Supp D$ and $\Supp B\subseteq \Supp D$.
\end{proposition}

\begin{proof}
Pushing forward the assumption,
we get $(\deg f)D\sim_{\Q} f_*D+f_*B$.
Thus, since $\kappa(X, f^*D) = \kappa(X, D) = 0$
(cf. \cite[Theorem 5.13]{Uen} or Lemma \ref{lem-kappa-pullback}), we have
$$f^{-1}(\Supp D) = (\Supp D) \cup (\Supp B)
\supseteq \Supp B, \,\,
\Supp D=(\Supp f_*D) \cup (\Supp f_*B).$$
Hence $\Supp f^{n+1}_*D\subseteq \Supp f^n_*D$, and by DCC we eventually get the equality.
Replacing $D$ by $f^{n}_*D$ we may assume $\Supp f_*D=\Supp D$.

Note that
$$\Supp f^*D=\Supp f^*f_*D=\Supp f_*f^*D=\Supp D.$$
The first equality is from $\Supp D=\Supp f_*D$, while
the second follows from $(\deg f) D =
f_*f^*D \sim_{\Q} f^*f_*D$ (cf.~Proposition \ref{prop-vf}) and $\kappa(X, D)=0$.
So $f^{-1}(\Supp D)=\Supp D$.
\end{proof}

\begin{theorem}\label{thm-kappa}
Let $X$ be a $\Q$-Gorenstein normal projective variety admitting an int-amplified endomorphism $f$.
Then we have:
\begin{itemize}
\item[(1)]
$-K_X\sim_{\mathbb{Q}} D$ for some effective $\Q$-Cartier divisor $D$.
\item[(2)]
Suppose further $\kappa(X,-K_X)=0$.
Then $D$ is an (integral) reduced effective Weil divisor; $\Supp R_f=\Supp D$ and it is $f^{-1}$-invariant.
Moreover, for any surjective endomorphism $g$ of $X$, we have $g^{-1}(D) = D$ and $\Supp R_g \subseteq \Supp D$, i.e.,
$g|_{X\backslash D}: X\backslash D \, \to \, X\backslash D $ is quasi-\'etale.
\end{itemize}
\end{theorem}

\begin{proof}
(1) We use the notation in Definition \ref{def-vf}.
By the ramification divisor formula,
$$f^*(-K_X)-(-K_X) = R_f\in \EFF_f(R_f).$$
Therefore, $-K_X\in \EFF_f(R_f)$ by Propositions \ref{prop-iamp->1}, \ref{prop-vf} and \cite[Proposition 3.2]{Meng_IAMP}.
This and $K_X$ being $\Q$-Cartier, imply
that $-K_X\sim_{\mathbb{Q}} D$ for some effective $\Q$-Cartier divisor $D$.

(2) Suppose $\kappa(X, -K_X)=0$.
By Proposition \ref{prop-kap0-inv}, $f^{-1}(\Supp D)=\Supp D$ and $\Supp R_f\subseteq \Supp D$.
Write $D=\sum a_i D_i$ where $D_i$ is the irreducible components of $D$ and $a_i>0$.
Replacing $f$ by a positive power, we may assume $f^{-1}(D_i)=D_i$.
Since $f$ is int-amplified, we have $f^*D_i=q_iD_i$ with $q_i>1$ (cf.~\cite[Theorem 1.1]{Meng_IAMP}).
So
$$\sum (q_i-1)D_i = R_f \sim_{\Q} f^*D-D=\sum a_i(q_i-1)D_i.$$
Since $\kappa(X, R_f)=0$ and $q_i>1$, we have $a_i=1$ for each $i$.
The last assertion of (2) follows from Proposition \ref{prop-kap0-inv}
since $g^*(-K_X) - (-K_X) = R_g \ge 0$.
\end{proof}

\section{Anti-Kodaira Fibration: Proof of Proposition \ref{main-thm-caseB}}\label{sec-kappa>0}

In this section, we focus on the case when $f^*K_X\equiv \delta_f K_X$ and $\kappa(X, -K_X)>0$.
We show that the Chow reduction of the Iitaka fibration $\pi:X\dashrightarrow Y$ of $-K_X$ is $f$-equivariant.
By some further cone analysis, we show that $f|_Y$ is $\delta_f$-polarized.

We first recall the definition and properties of the Chow reduction in \cite[Proposition 4.14 and Definition 4.15]{Nak-INS}, using the formulation in his RIMS preprint version.

\begin{proposition}[Chow reduction]\label{prop-nak}
Let $\pi:X\dashrightarrow Y$ be a dominant rational map from a projective variety $X$ to a normal projective variety $Y$.
Then there exist a normal projective variety T and a birational map $\mu:Y\dashrightarrow T$ satisfying the following conditions:
\begin{itemize}
\item[(1)] The graph $\gamma_{\mu\circ\pi}:\Gamma_{\mu\circ\pi} \to T$ of $\mu\circ\pi$ is equi-dimensional.
\item[(2)] Let $\mu': Y \dashrightarrow T'$ be a birational map to another normal projective variety $T'$ such that the graph $\gamma_{\mu'\circ\pi}:\Gamma_{\mu'\circ\pi}\to T'$ of $\mu'\circ\pi$
is equi-dimensional. Then there exists a birational morphism $\nu:T'\to T$ such that $\mu=\nu\circ\mu'$.
\end{itemize}
\end{proposition}

We call the composition $\mu \circ \pi: X \dashrightarrow T$ above satisfying Proposition
\ref{prop-nak} (1) - (2) the {\it Chow reduction} of
$\pi: X \dashrightarrow Y$, which is unique up to isomorphism.

Theorem \ref{thm-chow-equiv} below is a generalization of Nakayama \cite[Theorem 4.19]{Nak-INS} with exactly the same proof.
Note that his special MRC fibration there (also a Chow reduction) is used only to secure our following assumption that $g\circ\pi=\pi\circ f$ for some dominant self-map $g$ on the base of the Chow reduction  $\pi$, precisely, for him to show in \cite[Proof of Theorem 4.19, page 592, lines 4-9, after the display]{Nak-INS} that his $Y$ and $Y_1$ there are birational (to the same $W$ there) so that $f : X \to X$ descends to
a rational self-map $g : Y \dashrightarrow Y$.
Then his argument there further shows that $g$ is a surjective {\it endomorphism}.
His polarized assumption is only used to show that $g$ is polarized.

Recall that we say a dominant map $\pi:X\dashrightarrow Y$ of normal projective varieties has {\it connected fibres} if the normalization of the graph $\overline{\Gamma_\pi}\to Y$ has connected fibres.
In characteristic $0$, by the Stein factorization, it is equivalent to saying that the function field $k(Y)$ is algebraically closed in $k(X)$.

\begin{theorem}\label{thm-chow-equiv} (cf.~\cite[Theorem 4.19]{Nak-INS})
Let $\pi:X\dashrightarrow Y$ be a dominant map of normal projective varieties with connected fibres.
Let $f:X\to X$ be a surjective endomorphism and $g:Y\dashrightarrow Y$ a dominant self-map such that $g\circ\pi=\pi\circ f$.
Suppose $\pi$ is a Chow reduction of itself.
Then $g$ is a surjective endomorphism.
\end{theorem}

We now recall some fundamental results about Iitaka fibrations.
\begin{lemma}\label{lem-kappa-<=}
Let $X$ be a normal projective variety. Let $D_1$ and $D_2$ be two effective Cartier divisors with $D_2-D_1$ effective and $\kappa(X,D_1)=\kappa(X,D_2)$.
Then for $t\gg 1$, the Iitaka fibrations
$\phi_{tD_i}$ satisfy $\phi_{tD_1}=\sigma\circ \phi_{tD_2}$ for some birational map $\sigma$.
\end{lemma}

\begin{proof}
Let $s_0,\cdots, s_{m(1)}$ be a basis of $H^0(X, tD_1)$ and let $t_0,\cdots, t_{m(2)}$ be a basis of $H^0(X, tD_2)$ where $t_i = \xi s_i$ ($0 \le i \le m(1)$) with $div(\xi) = t(D_2 - D_1)$.
Define $p_1:X\dashrightarrow \mathbb{P}^{m(1)}$ via $p_1(x)=(s_0(x):\cdots:s_{m(1)}(x))$
and $p_2:X\dashrightarrow \mathbb{P}^{m(2)}$ via $p_2(x)=(t_0(x):\cdots:t_{m(2)}(x))$,
so that $p_i$ is the composition of the Iitaka fibration $\phi_{tD_i} : X \dashrightarrow Y_i$ and
embedding $Y_i \subseteq \PP^{m(i)}$.
Define $h:\mathbb{P}^{m(2)}\dashrightarrow \mathbb{P}^{m(1)}$ via $h(x_0:\cdots:x_{m(2)})=(x_0:\cdots:x_{m(1)})$.
Then $p_1=h\circ p_2$.
Since the Iitaka fibrations have connected fibres, there exists some dominant rational map $\sigma:Y_2\dashrightarrow Y_1$ with connected fibres such that $\phi_{tD_1}=\sigma\circ\phi_{tD_2}$ by the universal property of Stein factorization of $\phi_{tD_1}$.
Moreover, $\sigma$ is birational since $\dim(Y_1)=\dim(Y_2)$.
\end{proof}

\begin{lemma}\label{lem-iit-g'}
Let $f:X\to X$ be a surjective endomorphism of a normal projective variety.
Let $D$ be an effective Cartier divisor.
Let $\phi_{tD}:X\dashrightarrow Y$ and $\phi_{tf^*D}:X\dashrightarrow Y'$ be the Iitaka fibrations with $t\gg 1$.
Then $g'\circ \phi_{tf^*D}=\phi_{tD}\circ f$ for some dominant rational map $g':Y'\dashrightarrow Y$.
\end{lemma}

\begin{proof}
Let $\phi_{f^*|tD|}: X\dashrightarrow Z$ be the dominant rational map defined by $f^*|tD|$ where $|tD|$ is the complete linear system of $tD$.
Clearly, $Z=Y$ and $\phi_{f^*|tD|}=\phi_{tD}\circ f$.
Since $f^*|tD|$ is a sub linear system of $|tf^*D|$,
by the argument in the proof of Lemma \ref{lem-kappa-<=}, there is a dominant rational map $g':Y'\dashrightarrow Y$ such that $\phi_{f^*|tD|}=g'\circ \phi_{|tf^*D|}=g'\circ \phi_{tf^*D}$.
\end{proof}

We recall the following well-known useful result.

\begin{lemma}\label{lem-kappa-pullback}(cf.~\cite[Theorem 5.13]{Uen})
Let $f:X\to Y$ be a surjective morphism of projective varieties and let $D$ be a Cartier divisor of $Y$.
Then $\kappa(Y, D)=\kappa(X, f^*D)$.
\end{lemma}

\begin{corollary}\label{cor-iitaka-des}
Let $f:X\to X$ be a surjective endomorphism of a normal projective variety $X$ with $\kappa(X,-K_X)\ge 0$.
Let $\phi_{-mK_X}:X\dashrightarrow Y$ be the Iitaka fibration with $m\gg 1$.
Then there is a dominant self-map $g:Y\dashrightarrow Y$ such that $g\circ \phi_{-mK_X}=\phi_{-mK_X}\circ f$.
\end{corollary}

\begin{proof}
Let $\phi_{mf^*(-K_X)}:X\dashrightarrow Y'$ be the Iitaka fibration with $m\gg 1$.
By Lemma \ref{lem-iit-g'}, $g'\circ\phi_{mf^*(-K_X)}=\phi_{-mK_X}\circ f$ for some dominant rational map $g':Y'\dashrightarrow Y$.

By the ramification divisor formula, we have $f^*(-K_X)=-K_X+R_f$.
By Lemma \ref{lem-kappa-pullback}, $\kappa(X, f^*(-K_X))=\kappa(X, -K_X)$.
Then $\phi_{-mK_X}=\sigma\circ\phi_{mf^*(-K_X)}$ for some birational map $\sigma:Y'\dashrightarrow Y$ by Lemma \ref{lem-kappa-<=}.
Let $g:=g'\circ\sigma^{-1}$.
Then $g\circ \phi_{-mK_X}=\phi_{-mK_X}\circ f$.
\end{proof}

\begin{lemma}\label{lem-face-4}
Consider the following commutative diagram of normal projective varieties
$$\xymatrix{
W\ar[r]^{\sigma_W}\ar[d]^{\phi_W}&X\ar@{.>}[d]^{\phi_{mD}}\\
Y\ar@{.>}[r]^{\sigma_Y}&Z
}$$
where $\phi_{mD}$ is the Iitaka fibration of some effective Cartier divisor $D$ of $X$ with $m\gg 1$, $\sigma_W$ is a birational morphism, $\sigma_Y$ is a birational map, and $\phi_W$ is a surjective morphism.
Let $F \subseteq \PE^1(W)$ be the minimal extremal face containing $\sigma_W^*D$.
Then $\phi_W^*(\PE^1(Y))\subseteq F$.
\end{lemma}

\begin{proof}
Taking a sufficiently high resolution $i:W'\to W$, we have a birational morphism $\sigma_{W'}:W'\to X$ such that  $\sigma_{W'}^*|mD|=\mathfrak{d}+\Delta$ where $\mathfrak{d}$ is a free linear system and $\Delta$ is the fixed component.
Then $\phi_\mathfrak{d}=\phi_{mD}\circ \sigma_{W'}$.
Let $M\in \mathfrak{d}$.
Then $M=\phi_\mathfrak{d}^*A$ for some ample Cartier divisor $A$ on $Z$.

Consider the following commutative diagram
$$\xymatrix{
&W'\ar[r]^{i}\ar[d]^{j}\ar[dl]_{\phi_{\mathfrak{d}}}&W\ar[d]^{\phi_W}\\
Z&\widetilde{Y}\ar[l]_{p_1}\ar[r]^{p_2}&Y
}$$
where $\widetilde{Y}$ is the graph of $\sigma_Y$, $p_1$ and $p_2$ are the two (birational) projections,
and $j$ is a morphism induced by the two morphisms $\phi_W\circ i$ and $\phi_\mathfrak{d}$.

Let $H$ be an effective Cartier divisor of $Y$.
The class of $E':=p_1^*sA-p_2^*H$ is the class of an effective divisor for some $s\gg 1$.
Note that
$$\sigma_{W'}^*smD\sim sM+s\Delta=\phi_{\mathfrak{d}}^*sA+s\Delta=j^*p_1^*sA+s\Delta=j^*p_2^*H+j^*E'+s\Delta=i^*\phi_{W}^*H+j^*E'+s\Delta.$$
Taking the pushforward of $i$, we have
$$\sigma_W^*smD=\phi_W^*H+i_*(j^*E'+s\Delta).$$
Since $F$ is the minimal extremal face of $\PE^1(W)$ containing $\sigma_W^*D$, we have $\phi_W^*H\in F$.
Therefore, $\phi_W^*(\PE^1(Y))\subseteq F$.
\end{proof}

\begin{theorem}\label{thm-caseB}
Let $f:X\to X$ be a surjective endomorphism of a $\Q$-Goreinstein normal projective variety $X$ such that $f^*K_X\equiv qK_X$ for some integer $q>1$.
Suppose $\kappa(X,-K_X)>0$.
Then there is an $f$-equivariant dominant rational map $\pi:X\dashrightarrow Y$ to a normal projective variety $Y$ such that $\dim(Y)>0$ and $f|_Y$ is $q$-polarized.
\end{theorem}

\begin{proof}
Let $\phi_{mD}:X\dashrightarrow Z$ be the Iitaka fibration of $D:=-K_X$ with $m\gg 1$.
By Corollary \ref{cor-iitaka-des} and Theorem \ref{thm-chow-equiv}, there is a birational map $\sigma_Y: Y\dashrightarrow Z$ such that $\pi:=\sigma_Y^{-1}\circ\phi_{mD}$ is (the Chow reduction of $\phi_{mD}$ and) $f$-equivariant.
Denote by $g:=f|_Y$.

Let $W$ be the normalization of the graph of $\pi$.
We have the following commutative diagram
$$\xymatrix{
W\ar[r]^{\sigma_W}\ar[d]^{\phi_W}&X\ar@{.>}[d]^{\phi_{mD}}\\
Y\ar@{.>}[r]^{\sigma_Y}&Z
}$$
Let $F$ be the minimal extremal face containing $\sigma_W^*D$ in $\PE^1(W)$.
By Lemma \ref{lem-face-4}, we have $\phi_W^*(\PE^1(Y))\subseteq F$.

Note that $f$ lifts to a surjective endomorphism $h:W\to W$.
Since $h^*\sigma_W^*D\equiv q\sigma_W^*D$, we have $h^*(F)=F$ by the uniqueness of $F$ (cf.~\cite[Lemma 4.2]{Meng_PCD}).
Denote by $\langle F \rangle$ the subspace in $\N^1(W)$ spanned by $F$.
By \cite[Propositions 2.9]{MZ}, $h^*|_{\langle F \rangle}$ is diagonalizable with all the eigenvalues being of the same modulus $q$.
Therefore so is $g^*|_{\N^1(Y)}$ since $\N^1(Y)=\langle \PE^1(Y) \rangle \subseteq \langle F \rangle$.
By \cite[Propositions 2.9 and 1.1]{MZ}, $g$ is $q$-polarized.
\end{proof}

Now we can show Proposition \ref{main-thm-caseB} easily.
\begin{proof}[Proof of Proposition \ref{main-thm-caseB}]
We may assume $\delta_f>1$.
Then the theorem follows directly from Theorem \ref{thm-caseB}, Lemma \ref{lem-ksc-iff} and Proposition \ref{prop-ksc-kappa(D)>=0}.
\end{proof}

\section{Case TIR: Conditions (A1) - (A4) Imply Condition (A5)}

In this section, we show that in Case TIR, Conditions (A1)-(A4) imply Condition (A5).
The main idea is to take the double cover as in Step 5 of the proof of Theorem \ref{thm-surf-rho2}.

We first recall the result below.

\begin{lemma}\label{lem-nak-gkp}(cf.~\cite[Lemma 3.3.1]{ENS} or \cite[Lemma 2.3]{Nak-1923}, \cite[Lemma 2.5]{Na-Zh})
	Let $f:X\to X$ be a non-isomorphic surjective endomorphism of a normal projective variety $X$.
	Let $\theta_k:V_k\to X$ be the Galois closure of $f^k:X\to X$ for $k\ge 1$ and let $\tau_k:V_k\to X$ be the induced finite Galois covering such that $\theta_k=f^k\circ \tau_k$.
	Then there are finite Galois morphisms $g_k, h_k: V_{k+1}\to V_k$ such that $\tau_k\circ g_k=\tau_{k+1}$, $\tau_k\circ h_k=f\circ \tau_{k+1}$
and $(\deg h_k)/ (\deg g_k) = \deg f$.
\end{lemma}

The following result about
periodic subvarieties is another application of the technique used in the proofs of \cite[Theorem 3.3]{Na-Zh} and \cite[Theorem 5.2]{Meng_IAMP}.

\begin{theorem}\label{thm-log-qe} Let $f:X\to X$ be an int-amplified endomorphism of a normal projective variety $X$.
	Suppose $D:=\Supp R_f$ is $f^{-1}$-invariant and $X\backslash D$ is klt.
	Let $Z$ be an $f^{-1}$-periodic proper closed subvariety of $X$.
	Then $Z\subseteq \Supp R_f$.
\end{theorem}
\begin{proof}
   	It suffices for us to consider the case when $Z$ is irreducible.
	We apply Lemma \ref{lem-nak-gkp} and use the notation there.
Set $d:=\deg f$. Then $d = (\deg h_k)/(\deg g_k)$, and $d > 1$ (cf.~\cite[Lemma 3.7]{Meng_IAMP}).
Denote by $U_k:=V_k\backslash \tau_k^{-1}(D)$.
	Then $U_{k+1}=g_k^{-1}(U_k)=h_k^{-1}(U_k)$.
	By the ramification divisor formula, $f|_{X\backslash D}: X\backslash D\to X\backslash D$ is quasi-\'etale.
	Hence $\theta_k|_{U_k}$, $\tau_k|_{U_k}$, $g_k|_{U_{k+1}}$ and $h_k|_{U_{k+1}}$ are quasi-\'etale and Galois by the construction.
So $U_k$ is klt by \cite[Proposition 5.20]{KM}.
	Therefore, $g_k|_{U_{k+1}}$ and $h_k|_{U_{k+1}}$ are \'etale for $k\gg 1$ by \cite[Theorem 1.1]{GKP}.	
	Let $A$ be an ample Cartier divisor on $X$.
	Denote by $A_{k}:=\tau_{k}^*A$ and $(f^*A)_{k}:=\tau_{k}^*(f^*A)$.
	Denote by $S_k:=\tau_k^{-1}(Z)$.
	In the rest of the proof, we always assume $k\gg 1$.
	
	Suppose $Z\not\subseteq D$.
	Then $S_{k+1}=g_k^{-1}(S_k)=h_k^{-1}(S_k)$. Note that $g_k$ and $h_k$ are \'etale over the generic point of $S_k$. By viewing $S_k$ as a cycle with simply the reduced structure, we have $S_{k+1}\equiv_w g_k^*S_k\equiv_w h_k^*S_k$; see \cite[\S 2.3]{Zh-comp} for the pullback of cycles by finite surjective morphisms.
	Let $m=\dim(Z)<\dim X$.
	By the projection formula, we have
	$$S_{k+1}\cdot(f^*A)_{k+1}^m=S_{k+1}\cdot g_k^*((f^*A)_k)^m=(\deg g_k)  S_k\cdot (f^*A)_k^m$$
	and
	$$S_{k+1}\cdot(f^*A)_{k+1}^m=S_{k+1}\cdot h_k^*(A_k)^m=(\deg h_k) S_k\cdot A_k^m.$$
	Then $S_k\cdot (f^*A)_k^m=d S_k\cdot A_k^m$.
	Note also that $(\tau_k)_{\ast}S_k=t_kZ$ for some integer $t_k>0$.
	Thus, by the projection formula, we have $t_kZ\cdot (f^*A^m)=d t_kZ\cdot A^m$.
	Therefore,
	$$1\le Z\cdot A^m=\lim\limits_{i\to+\infty}Z\cdot \frac{(f^i)^*A^m}{d^i}=0$$ with the last equality by \cite[Lemma 3.8]{Meng_IAMP}, a contradiction.
\end{proof}

\begin{lemma}\label{lem-cycl_pullback}
Let $\pi: X \to Y$ be a surjective morphism of normal projective varieties with $Y$ being $\Q$-factorial.
Let $D_1,\cdots, D_r\in \NS_{\C}(Y)$ such that $D_1\cdots D_r\equiv_w 0$ (weak numerical equivalence).
Then $\pi^*D_1\cdots \pi^*D_r \equiv_w 0$.
\end{lemma}

\begin{proof}
Let $n = \dim(X) \ge m: = \dim(Y)$ and $d:=n-m$.
Suppose the contrary that
$\pi^*D_1\cdots \pi^*D_r \not\equiv_w 0$.
Then we can find (general) very ample divisors $H_i$ of $X$ such that $H_1\cdots H_{n-r}\cdot \pi^*D_1\cdots \pi^*D_r\neq 0$.
Since $X$ is normal, we may assume $H_1$ is a normal
variety (cf.~\cite{Sei}).
Inductively, by the Bertini's theorem, we may assume that each
$Z_s:=H_1\cap \cdots \cap H_s$ ($1 \le s \le n-r$) is an irreducible normal subvariety (and a Cartier divisor) of $Z_{s-1}$ with $\dim(Z_s)= n-s\ge r\ge 1$, and $\pi|_{Z_s}:Z_s\to \pi(Z_s)$ ($s \ge d$) is generically finite.
Then $\pi(Z_{d})=Y$ and $\pi(Z_s)= \pi(Z_{s-1} \cap H_s) = \bigcap_{d< i \le s}\pi(H_i \cap Z_{d})$
for $s > d$.
Note that $H'_s:=(\pi|_{Z_{d}})_*(H_s|_{Z_{d}})$ ($s > d$) is a $\Q$-Cartier divisor on $Y$
since $Y$ is $\Q$-factorial. In particular, $\pi_*(H_1\cdots H_{n-r})=e H'_{d+1}\cdots H'_{n-r}$ for some $e>0$.
Note that $n-r\ge d$.
By the projection formula:
$$0\neq H_1\cdots H_{n-r}\cdot \pi^*D_1\cdots \pi^*D_r =\pi_*(H_1\cdots H_{n-r})\cdot D_1\cdots D_r =e H'_{d+1} \cdots H'_{n-r} \cdot D_1\cdots D_r,$$
contradicting that $D_1\cdots D_r \equiv_w 0$.
\end{proof}

\begin{proposition}\label{prop-D^d+1=0}
	Let $f:X\to X$ be a surjective endomorphism of a $\Q$-factorial lc projective variety $X$.
	Let $\pi:X\to Y$ be an $f$-equivariant Fano contraction with general fibre $F$.
	Suppose $\delta_f>\delta_{f|_Y}$ and $f^*D\equiv \delta_f D$ for some $\pi$-ample $D\in \N^1(X)$.
	Then $D^d\not\equiv_w 0$ and $D^{d+1}\equiv_w 0$ (weak numerical equivalence) with $d:=\dim(X)-\dim(Y)$.
\end{proposition}

\begin{proof}
Note that $D|_F$ is ample, hence $D^d\cdot F=(D|_F)^d>0$.
We identify $\NS_{\mathbb{C}}(Y)$ as a subspace of  $\NS_{\mathbb{C}}(X)$ via $\pi^*$.
Set $x :=D$.
Then $\NS_{\mathbb{C}}(X)$ is spanned by $\NS_{\mathbb{C}}(Y)$ and $x$.
	
Let $\{y_{i_j}\}_{1\le i\le k, 1\le j\le \ell_i}$ be a basis of $\NS_{\mathbb{C}}(Y)$ such that $g^*(y_{i_j})=\lambda_i y_{i_j}+ y_{i_{j+1}}$ if $j<\ell_i$, and $g^*(y_{i_j})=\lambda_i y_{i_j}$ if $j=\ell_i$.
	For two sequences of integers, we say $(a_{i_j})< (b_{i_j})$ if for some $i'$ and $j'$, $a_{i'_{j'}}<b_{i'_{j'}}$ and $a_{i_j} \le b_{i_j}$ when $i>i'$ and when $i=i'$ and $j>j'$.
	Let $s\ge d$ be the maximal integer such that $x^s\not\equiv_w 0$.
	Let $(a_{i_j})$ be the maximal sequence such that $\sum_{i_j} a_{i_j}=\dim(X)-s$
	and $x^s\cdot\prod_{i_j} y_{i_j}^{a_{i_j}}\neq 0$.
	For convenience, we call $(a_{i_j})$ the degree sequence of $y_{i_j}^{a_{i_j}}$ and $\sum_j a_{i_j}$ the $i$-th degree of $(a_{i_j})$.
	
	Note that $$f^*(x^s\cdot\prod_{i_j} y_{i_j}^{a_{i_j}})=\delta_f^sx^s\cdot \{\prod_{i_j} (\lambda_i y_{i_j})^{a_{i_j}}+\Delta\}$$
	where the degree sequence of each term of $\Delta$  is larger than $(a_{i_j})$.
	Thus $x^s\cdot \Delta=0$, so
	$$\deg f=\delta_f^s\cdot \prod_i \lambda_i^{\sum_j a_{i_j}}.$$
	
Lemma \ref{lem-cycl_pullback} implies $\prod_{i_j} y_{i_j}^{a_{i_j}}\not\equiv_w 0$ in $Y$, noting that $Y$ is $\Q$-factorial (cf.~\cite[Corollary 3.18]{KM}).
So $\prod_{i_j} y_{i_j}^{a_{i_j}+b_{i_j}}\neq 0$ for some $b_{i_j}\ge 0$ and $\sum_{i_j} (a_{i_j}+b_{i_j})=\dim(Y)$.
	Let $(c_{i_j})$ be the maximal sequence such that $\sum_j c_{i_j}=\sum_j (a_{i_j}+b_{i_j})$
for each $i$ and $\prod_{i_j} y_{i_j}^{c_{i_j}}\neq 0$.
	Note that $$g^*(\prod_{i_j} y_{i_j}^{c_{i_j}})=\prod_{i_j} (\lambda_i y_{i_j})^{c_{i_j}}+\Delta'$$
	where the degree sequence of each term of $\Delta'$  is larger than $(c_{i_j})$ and the $i$-th degree of each term of $\Delta'$ is still $\sum_j (a_{i_j}+b_{i_j})$ for each $i$.
	Then $\Delta'=0$ and hence
	$$\deg g=\prod_i \lambda_i^{\sum_j c_{i_j}}.$$
	
	We may write $\prod_{i_j} y_{i_j}^{c_{i_j}}\equiv tF$ on $X$ for some $0 \ne t \in \C$.
	Since $D|_F$ is ample, we have $x^d\cdot \prod_{i_j} y_{i_j}^{c_{i_j}}\neq 0$.
	Note that $$f^*(x^d\cdot\prod_{i_j} y_{i_j}^{c_{i_j}})=\delta_f^dx^d\cdot \prod_{i_j} (\lambda_i y_{i_j})^{c_{i_j}}.$$
	Then $$\deg f=\delta_f^d\cdot \prod_i \lambda_i^{\sum_j c_{i_j}}.$$
	
	Finally, we have $\delta_f^{s-d}=\prod_i \lambda_i^{\sum_j (c_{i_j}-a_{i_j})}$.
Since $\sum_j (c_{i_j}-a_{i_j})\ge 0$ for each $i$, $\sum_{i_j} (c_{i_j}-a_{i_j}) = \dim (Y) - (\dim (X) - s) = s-d$
and $|\lambda_i|\le \delta_g<\delta_f$,
	we have $s=d$.
\end{proof}

\begin{lemma}\label{lem-fin-qe}
	Let $\pi:X\to Y$ be a degree two finite surjective morphism of normal varieties.
	Let $f:X\to X$ and $g:Y\to Y$ be surjective endomorphisms such that $\pi\circ f=g\circ \pi$.
	Suppose $g$ is quasi-\'etale, and there is no $g^{-1}$-periodic prime divisor of $Y$.
	Then $\pi$ and $f$ are quasi-\'etale.
\end{lemma}

\begin{proof}
Suppose prime divisor $Q_1$ of $Y$ is in $B_{\pi}$, the branch locus of $\pi$. Then $\pi^{-1}(Q_1) = P_1$
and $\pi^*Q_1 = 2P_1$, where $P_1$ is a prime divisor of $X$. Now $\pi\circ f=g\circ \pi$ implies
$2f^*(P_1) = \pi^*g^*(Q_1)$.
Thus $g^{-1}(Q_1) \subseteq B_{\pi}$
since $g$ is quasi-\'etale. So the set $g^{-1}(B_{\pi})$ is contained in the set $B_{\pi}$. Hence these two sets are the same since $g$ is surjective. We then have $B_{\pi} = 0$, by the assumption. Thus, $\pi$ and hence $g\circ \pi
= \pi\circ f$ and also $f$ are quasi-\'etale.
\end{proof}

\begin{theorem}\label{thm-tir-1}
	In Case TIR, Conditions (A1)-(A4) imply Condition (A5).
\end{theorem}

\begin{proof}
We assume (A1) - (A4).
We will deduce (A5).
If $\dim(Y)=0$, then $X$ is a klt Fano variety,
so $\kappa(X, -K_X)>0$, contradicting Condition (A1). Thus $\dim Y \ge 1$.
	
We still have to consider the case $\dim(X)=\dim(Y)+1$.
Let $\mathcal{I}:X\to X$ be an int-amplified endomorphism. We may assume $\pi$ is $\mathcal{I}$-equivariant, after $\mathcal{I}$ is replaced by a positive power (cf.~\cite[Theorem 1.10]{Meng_IAMP}).
By Theorem \ref{thm-kappa}, $\mathcal{I}^{-1}(D)=D$ and $\Supp R_{\mathcal{I}}=D$.
	
	We first claim that $\pi|_D:D\to Y$ is finite.
Since $D \sim_{\Q} -K_X$ is $\pi$-ample and $\dim (D) = \dim (Y)$, $\pi|_D$ is generically finite.
	If $\pi|_D$ is not finite, then $D$ contains some curve $C$ contracted by $\pi$.
	Since $D$ is $\pi$-ample, $D\cdot C>0$.
	However, $D^2\equiv_w 0$ by Proposition \ref{prop-D^d+1=0}.
	So $D|_D\equiv 0$ (cf.~\cite[Lemma 3.2]{Zh-tams}) and hence $D\cdot C=D|_D\cdot C=0$, a contradiction.
The claim is proved.
	
	Let $\widetilde{X}$ be the normalization of the (irreducible) main component of $X\times_Y \widetilde{Y}$ where $\widetilde{Y}$ is the normalization of $D$.
	Denote by $p_1:\widetilde{X}\to X$ and $p_2:\widetilde{X}\to \widetilde{Y}$ the induced projections.
	Denote by $\widetilde{f}:\widetilde{X}\to \widetilde{X}$ and $\widetilde{\mathcal{I}}:\widetilde{X} \to \widetilde{X}$ the equivariant liftings of $f$ and $\mathcal{I}$.
Set $\widetilde{D}:=p_1^{-1}(D)$.
	Since the general fibre $F$ of $\pi$ is $\PP^1$, we have $K_X\cdot F=-2$ and $D\cdot F=2$.
Since there is a diagonal embedding $D\to D\times_Y D$, our $\widetilde{D}$ is reducible. Write $\widetilde{D}:= \sum_{i=1}^2 \, \widetilde{D}_i$ with
$\widetilde{D}_i$ irreducible.
Replacing $\widetilde{f}$ and $\widetilde{\mathcal{I}}$ by positive powers, we may assume $\widetilde{D}_i$ is $\widetilde{f}^{-1}$ and $\widetilde{\mathcal{I}}^{-1}$-invariant for each $i$.
	
	Note that $p_1$ is a double cover and $\mathcal{I}|_{X\backslash D}$
is quasi-\'etale.
By Lemma \ref{lem-fin-qe}, $p_1|_{\widetilde{X}\backslash \widetilde{D}}$ and $\widetilde{\mathcal{I}}|_{\widetilde{X}\backslash \widetilde{D}}$ are  quasi-\'etale.
Since $\widetilde{D}$ has two irreducible components, $p_1^*D=\widetilde{D}$ and $p_1$ is quasi-\'etale.
Then $K_{\widetilde{X}}=p_1^*K_X$ and $\widetilde{D}$ are $\Q$-Cartier.
By \cite[Proposition 5.20]{KM}, $\widetilde{X}$ is klt.
So $\Supp R_{\widetilde{\mathcal{I}}}=\widetilde{D}$.
Further, $p_2$ has connected fibres and $p_2|_{\widetilde{D}}:\widetilde{D}\to \widetilde{Y}$ is a finite surjective morphism since so is $\pi|_D$.
	
Note that $\widetilde{D}_1\cap \widetilde{D}_2$ is $\widetilde{\mathcal{I}}^{-1}$-invariant closed and does not dominate $\widetilde{Y}$.
	Replacing $\widetilde{\mathcal{I}}$ by a positive power, $p_2(\widetilde{D}_1\cap \widetilde{D}_2)$ is $\widetilde{\mathcal{I}}|_{\widetilde{Y}}^{-1}$-invariant by \cite[Lemma 7.5]{CMZ}.
	Let $Z:=p_2^{-1}(p_2(\widetilde{D}_1\cap \widetilde{D}_2))$ which is $\widetilde{\mathcal{I}}^{-1}$-invariant.
	By Theorem \ref{thm-log-qe}, $Z\subseteq \widetilde{D}$.
Then $\widetilde{D}_1\cap \widetilde{D}_2=\emptyset$,
since $\widetilde{D}$ contains no fibre of $p_2$.
Note that $\widetilde{D}$ is $\Q$-Cartier.
The non-$\Q$-Cartier locus of $\widetilde{D}_1$ is contained in $\widetilde{D}_1\cap \widetilde{D}_2$.
	So $\widetilde{D}_i$ is $\Q$-Cartier and $\widetilde{f}^*\widetilde{D}_i=\delta_f \widetilde{D}_i$ for each $i$ .
	
Since the general fibre of ${\widetilde{X}} \to {\widetilde{Y}}$ is still $\PP^1$,
$K_{\widetilde{X}}$ is not pseudo-effective over $\widetilde{Y}$.
	By the relative cone theorem (cf.~\cite[Theorem 3.25]{KM} and \cite[Theorem 1.1]{MZ_PG}), replacing $\widetilde{\mathcal{I}}$ by a positive power, there is an $\widetilde{\mathcal{I}}$-equivariant contraction $\pi_{\widetilde{C}}: \widetilde{X}\to B$ over $\widetilde{Y}$ of some $K_{\widetilde{X}}$-negative extremal ray $R_{\widetilde{C}}$.
	If $\pi_{\widetilde{C}}$ is birational with $E$ the exceptional locus, then $p_2(E)\subsetneq \widetilde{Y}$ is $\widetilde{\mathcal{I}}|_{\widetilde{Y}}^{-1}$-invariant by \cite[Lemma 7.5]{CMZ} and hence $p_2^{-1}(p_2(E))$ is $\widetilde{\mathcal{I}}^{-1}$-invariant.
	By Theorem \ref{thm-log-qe}, $p_2^{-1}(p_2(E))\subseteq \widetilde{D}$, a contradiction since $\widetilde{D}$ does not contain any fibre of $p_2$.
	So $\dim(\widetilde{X}) -1 = \dim(\widetilde{Y}) \le \dim(B)<\dim(\widetilde{X})$. Thus the induced morphism $\pi_B:B\to \widetilde{Y}$ is generically finite and hence birational since $p_2$ has connected fibres.
	Similarly, $\pi_B$ has to be isomorphic.
	So $p_2$ is a Fano contraction.
	
	Note that $\widetilde{D}_i$ is $p_2$-ample.
	Then for some rational number $t>0$, $\widetilde{D}_1-t\widetilde{D}_2\in p_2^*(\Pic_{\Q}(\widetilde{Y}))$ by the cone theorem (cf.~\cite[Theorem 3.7]{KM}).
	Denote by $\widetilde{g}:=\widetilde{f}|_{\widetilde{Y}}$.
	Then $\widetilde{g}^*(\widetilde{D}_1-t\widetilde{D}_2)=\delta_f(\widetilde{D}_1-t\widetilde{D}_2)$.
	Note that $\delta_{\widetilde{g}}=\delta_{f|_Y}<\delta_f = \delta_{\widetilde f}$ (cf.~Lemma \ref{lem-d1-bir}).
	Since $\widetilde{\mathcal{I}}|_{\widetilde{Y}}$ is int-amplified (cf.~\cite[Lemma 3.4]{Meng_IAMP}), the Albanese morphism of $\widetilde{Y}$ is surjective by \cite[Theorem 1.8]{Meng_IAMP}.
	So $\widetilde{D}_1-t\widetilde{D}_2\sim_{\Q} 0$ by Proposition \ref{prop-eig-sqrt}.
	Therefore, $\kappa(\widetilde{X},\widetilde{D}_1)>0$ and hence $\kappa(X,-K_X)=\kappa(X,D)>0$. This contradicts (A1). Thus (A5) holds.
\end{proof}

\section{Reduction to Case TIR: Proof of Theorem \ref{main-thm-tir}}

The following result is simple but useful.

\begin{lemma}\label{lem-ray-ext} Let $f:V\to V$ be an invertible linear map of a finite dimensional normed real vector space $V$ such that $f(C)=C$ for a closed convex cone $C\subseteq V$ which spans $V$ and contains no line.
	Suppose $f(x)=qx$ for some $0\neq x\in C$ and $q>0$.
	Suppose further that $q$ is the only eigenvalue of $f$ which has modulus $q$ and the $q$-eigenspace is $1$-dimensional.
	Then the ray $R_x$ generated by $x$ is extremal in $C$.
\end{lemma}

\begin{proof}
	Let $F$ be the minimal extremal face containing $x$ and $W$ the space spanned by $F$.
	Then $f(F)=F$ and $f(W)=W$ by (cf.~\cite[Lemma 2.7]{MZ}).
	By \cite[Lemma 4.2]{Meng_PCD} and \cite[Proposition 2.9]{MZ}, all the eigenvalues of $f|_W$ are of modulus $q$.
	So $\dim(W)=1$ by the assumption.
	In particular, $F=R_x$ is an extremal ray of $C$.
\end{proof}

The following is the key in the proof of Theorem \ref{main-thm-tir} for the induction purpose.

\begin{proposition}\label{lem-fano}
	Let $X$ be a $\Q$-factorial klt projective variety.
	Let $\pi:X\to Y$ be a Fano contraction
	(so $Y$ is still $\Q$-factorial klt). Let $f:X\to X$ and $g:Y\to Y$ be surjective endomorphisms such that $g\circ \pi=\pi\circ f$.
	Suppose $\kappa(X,-K_X)\ge 0$ and any finite sequence of MMP starting from $X$ is $f$-equivariant after replacing $f$ by a positive power.
	Suppose further the Albanese morphism of $X$ is surjective.
	Then replacing $f$ by a positive power, there is an $f$-equivariant MMP:  a composition of birational MMP $X\dashrightarrow X'$ followed by a Fano contraction $X'\to Y'$, such that after replacing $(f, X, g, Y)$ by $(f|_{X'}, X', f|_{Y'}, Y')$, one of the following holds.
	\begin{itemize}
		\item[(1)] $f^*K_X\equiv \delta_f K_X$ with $\delta_f>1$ being an integer and $\kappa(X,-K_X)>0$.
        \item[(2)] $\delta_f>\delta_g$; $\kappa(X, -K_X)=0$, so $-K_X \sim_{\Q} D \ge 0$; the class of $-K_X$ is extremal in both the cone $\Nef(X)$ and the cone $\PE^1(X)$;  and $D=\Supp R_f$ is a prime divisor with $f^*D= \delta_f D$.
		\item[(3)] $\dim(Y)<\dim(X)$ and $\delta_{g}=\delta_{f}$.
	\end{itemize}
\end{proposition}

\begin{proof}
We show by induction on $\rho(X)$.
If $\rho(X)=1$ (i.e. $Y$ is a point and $X$ is Fano), then we have Case (1). So we assume $\rho(X) \ge 2$.
If $\delta_f=\delta_g$, then we have Case (3) with no replacement.
So it suffices to consider the case when $\delta_f>\delta_g$.

Note that $\N^1(X)/\pi^*\N^1(Y)$ is 1-dimensional and $f^*|_{\N^1(X)/\pi^*\N^1(Y)} = q \id$ for some integer $q>0$.
	Then $q = \delta_f$ and it is the only eigenvalue of $f^*|_{\N^1(X)}$ with modulus $\delta_f$ and the $q$-eigenspace is $1$-dimensional.
	By a version of the Perron-Frobenius theorem (cf. \cite{Birk}), $f^*D\equiv \delta_f D$ for some nef and $\pi$-ample Cartier divisor $D\in \N^1(X)$.
	Moreover, the ray $R_D$ generated by $D$ in $\N^1(X)$ is extremal in both $\Nef(X)$ and $\PE^1(X)$ by Lemma \ref{lem-ray-ext}.
	Let $a>0$ such that $B:=D+aK_X$ satisfies $B\cdot C=0$, where $C$ is a (rational) curve so that $R_C$ is the extremal ray of $\NE(X)$ contracted by $\pi$.
Then $B \in \pi^*\N^1(Y)$ by the cone theorem (cf.~\cite[Theorem 3.7]{KM}).
We discuss according to two situations on whether $B$ pseudo-effective or not.
	
	\textbf{Situation 1. $B$ is pseudo-effective.}
	Since $-K_X$ is effective, $D=B+(-aK_X)$ implies that the rays $R_D=R_B=R_{-K_X}$.
	In particular, $f^*K_X\equiv \delta_f K_X$ and $-K_X$ is extremal in both $\Nef(X)$ and $\PE^1(X)$.
	By the assumption $\kappa(X,-K_X)\ge 0$, we can assume $D$ is an effective $\Q$-divisor such that $-K_X\sim_{\Q} D$.
	If $\kappa(X,-K_X)>0$, we have Case (1).
	
	If $\kappa(X,-K_X)=0$, write $D=\sum a_i D_i$ with $a_i>0$ being a rational number and $D_i$ being irreducible.
Since $-K_X$ is extremal in $\PE^1(X)$, we have the rays $R_{D_i}=R_{-K_X}$.
	Applying Proposition \ref{prop-kap0-inv} to $-K_X$, we have $f^{-1}(D_i)=D_i$ for each $i$ after replacing $f$ by a positive power.
	Since $f^*D_i\equiv \delta_f D_i$ and $D_i$ is not numerically trivial, $f^*D_i=\delta_f D_i$.
	Suppose $\Supp D$ is reducible.
	Then $sD_1-tD_2\in \Pic^0(X)$ for some positive integers $s$ and $t$.
	Note that $f^*(sD_1-tD_2)=\delta_f(sD_1-tD_2)$ and $\delta_f>1$.
	Since the Albanese morphism of $X$ is surjective by the assumption, we have $sD_1-tD_2\sim_{\Q} 0$ by Proposition \ref{prop-eig-sqrt}.
	Therefore, $\kappa(X,-K_X)\ge \kappa(X, D_1)>0$, a contradiction.
	By Proposition \ref{prop-kap0-inv}, $\Supp R_f\subseteq D_1:=\Supp D$.
	On the other hand, $f^*D_1=\delta_f D_1$ with $\delta_f>1$ implies that $D_1 \subseteq \Supp R_f$.
	So $R_f=(\delta_f-1)D_1$ and hence $K_X+D_1=f^*(K_X+D_1)$.
	Note that $K_X+D_1\sim_{\Q} -D+D_1=(a_1-1)D_1$.
	So $a_1=1$ and $D=D_1$.
	So we have Case (2).
	
	\textbf{Situation 2. $B$ is not pseudo-effective.}
	For a small effective ample $\Q$-Cartier divisor $E$, $(1/a)B + E$ is not pseudo-effective.
	Denote by $A := E + (1/a)D$ which is ample since $D$ is nef.
	Thus $K_X+A = (1/a)B + E$ is not pseudo-effective.
	By \cite[Corolllary 1.3.3]{BCHM}, we may run $\varphi:X\dashrightarrow X'$, a birational $(K_X+A)$-MMP with scaling of $A$ (cf.~\cite[Section 3.10]{BCHM}) and end up with a Fano contraction $\pi':X'\to Y'$ of some $(K_{X'}+ A')$-negative extremal ray $R_{C'}$ where $A'$ is the strict transform of $A$.
	Note that this particular choice of $(K_X+A)$-MMP is also a $K_X$-MMP.
	By the assumption, replacing $f$ by a positive power, we may assume this MMP is $f$-equivariant.
	If $\rho(X')<\rho(X)$, then we are done by induction (noting that $X' \to \Alb(X') = \Alb(X)$
	is still surjective).
	If $\rho(X')=\rho(X)$, then $\varphi$ consists of only flips.
	Hence, we can use $\varphi^*$ to identify $\N^1(X')$ with $\N^1(X)$.
	Let $D'$ and $E'$ be the strict transform of $D$ and $E$.
	Since the curves in $R_{C'}$ cover $X'$, we have $E'\cdot C'\ge 0$.
	Then $$(K_{X'}+ \frac{1}{a}D')\cdot C'=(K_{X'}+ A')\cdot C'- E'\cdot C'<0.$$
So $K_{X'}+ \frac{1}{a}D'$ (identified with
$K_{X}+ \frac{1}{a}D = (1/a)B \in \pi^*\N^1(Y)$)
is not in $(\pi'\circ\varphi)^*\N^1(Y')$.
Thus $\pi^*\N^1(Y)$ and  $(\pi'\circ\varphi)^*\N^1(Y')$ are two different $f^*$-invariant hyperplanes of $\N^1(X)$.
	Note that $f^*|_{\N^1(X)}$ has only one eigenvalue of modulus $\delta_f$ and $\delta_f>\delta_g$.
	Then $\delta_f=\delta_{f|_{Y'}}$.
	So we have Case (3) after replacing $(X, Y)$ by $(X', Y')$.
\end{proof}

\begin{proof}[Proof of Theorem \ref{main-thm-tir}]
If $K_X$ is pseudo-effective, then $X$ is $Q$-abelian by \cite[Theorem 1.9]{Meng_IAMP} (without using the $\Q$-factorial condition on $X$).
So (1) follows from Theorem \ref{thm-ksc-qa}.

For (2), we show by induction on $\dim(X)$.
Since KSC holds for curves, assume
$\dim (X) \ge 2$.
By (1), we may assume $K_X$ is not pseudo-effective.
	Let $\mathcal{I} : X \to X$ be an int-amplified endomorphism.
	By \cite[Theorem 1.2]{MZ_PG}, replacing $f$ and $\mathcal{I}$ by positive powers, we may run $f$ and $\mathcal{I}$-equivariant MMP
	$$X=X_1\dashrightarrow \cdots \dashrightarrow X_i \dashrightarrow \cdots \dashrightarrow X_r\to X_{r+1} = Y,$$
	where $X_i\dashrightarrow X_{i+1}$ ($i \le r$) is birational, $\pi: X_r\to Y$ is a Fano contraction, each $X_j$ ($j \le r+1$) is still $\Q$-factorial klt,
and the descending of $\mathcal{I}$ to each $X_j$ is still int-amplified.
	By Lemma \ref{lem-ksc-iff}, we may replace $X$ by $X_r$.
		
	Note that any finite sequence of MMP starting from $X$ is $f$ and $\mathcal{I}$-equivariant after iterations by \cite[Theorem 1.1]{MZ_PG}, and $\kappa(X, -K_X)\ge 0$ by Theorem \ref{main-thm-kappa}.
	Moreover, the Albanese morphism of $X$ is surjective by \cite[Theorem 1.8]{Meng_IAMP}.
	So we may apply Proposition \ref{lem-fano} and it suffices for us to  consider the three cases there by Lemma \ref{lem-ksc-iff}.
	For Case (1), we are done by Proposition \ref{main-thm-caseB}.
	For Case (2), it is further Case TIR by Theorems \ref{thm-kappa} and \ref{thm-tir-1}; by the assumption, KSC holds for $f$.
	For Case (3), we may replace $X$ by a lower dimensional one and we are done by induction (cf.~Lemma \ref{lem-ksc-iff}).
\end{proof}

%
%

%
%
%
%

\section{Toric Characterizations and Proof of Theorem \ref{main-thm-src3}}\label{sec-toric}

In this section, we show that Case TIR$_3$ will not happen during any MMP starting from a rationally connected smooth projective threefold which admits an int-amplified endomorphism.
The key of the proof is a characterization of a toric pair in the presence of an int-amplified endomorphism with totally invariant ramification.

Recall that a normal projective variety $X$ over $k$ is said to be {\it toric} or a {\it toric variety}
if $X$ contains an algebraic torus $T = (k^*)^n$ as an (affine) open dense subset such that
the natural multiplication action of $T$ on itself extends to an action on the whole variety $X$.
In this case, let $D:=X\backslash T$, which is a divisor; the pair $(X, D)$ is said to be a {\it toric pair}.

We mainly focus on the following question in this section.
\begin{question}\label{que-toric}
	Let $X$ be a rationally connected smooth projective variety and $D\subset X$ a reduced divisor.
	Suppose $f:X\to X$ is an int-amplified endomorphism such that
$f^{-1}(D)=D$ and $f|_{X\backslash D} : X\backslash D
\to X\backslash D$ is quasi-\'etale.
	Is $(X,D)$ a toric pair?
\end{question}

First, Question \ref{que-toric} has been affirmatively answered when $X$ is a Fano manifold of Picard number $1$.
Indeed, Hwang and Nakayama showed then that $X$ is isomorphic to $\mathbb{P}^n$ and $D$ is a simple normal crossing divisor consisting of $n + 1$ hyperplanes;
see \cite[Theorem 2.1]{HN}.
Later, their result was generalized by the authors \cite[Corollary 1.4]{MZ_toric}, answering the above question affirmatively when $f$ is polarized.

We sketch the idea of the proof when $f$ is polarized.
A key step is in applying the dynamical property of $f$ to verify that the reflexive sheaf of germs of logarithmic 1-forms $\hat{\Omega}_X^1(\log D)$ (cf.~\cite[2.1]{MZ_toric})
is free, i.e., isomorphic to $\mathcal{O}_X^{\oplus n}$ where $n=\dim(X)$; see \cite[Proposition 2.3]{HN} and \cite[Theorem 5.4]{MZ_toric}.
Thus $h^0(X,\hat{\Omega}_X^1(\log D))=\dim(X)$.
The remaining steps do not involve $f$ at all.
Write $D=\sum_{i=1}^\ell D_i$ with $D_i$ irreducible.
Then one calculates by \cite[Theorem 4.5 and Remark 4.6]{MZ_toric} the complexity of the pair $(X,D)$ as
$$c(X,D):=\dim(X)+r(D)-\ell(D)= \dim(X)+h^1(X,\mathcal{O}_X)-h^0(X,\hat{\Omega}_X^1(\log D))=0$$
where $\ell(D):=\ell$ and $r(D)$ is the {\it rank} of the vector space spanned by $D_1,\cdots, D_\ell$ in $\N^1(X)$.
Finally, $(X,D)$ is a toric pair by the complexity criterion \cite[Theorem 1.2]{BMSZ}.

Thus, to fully answer Question \ref{que-toric}, we only need to generalize the above key step to the int-amplified case.
Imitating the proof of \cite[Proposition 2.3]{HN} and \cite[Theorem 5.4]{MZ_toric},
we just need to verify the following two conditions for some ample Cartier divisor $H$:
\begin{itemize}
\item[(i)]
$c_1(\hat{\Omega}_X^1(\log D))\cdot H^{n-1}= c_1(\hat{\Omega}_X^1(\log D))^2\cdot H^{n-2}= c_2(\hat{\Omega}_X^1(\log D))\cdot H^{n-2}= 0$.
\item[(ii)]
$\hat{\Omega}_X^1(\log D)$ is $H$-slope semistable.
\end{itemize}
We will see late on that the second condition is not easy to verify and remains unprovable for the general int-amplified case.
For the easy comparison with the polarized case, we will also consider the singular case.

We need the following to show the vanishing of $c_2(\hat{\Omega}_X^1(\log D))$.

\begin{proposition}\label{codim3-prop}(cf.~\cite[Proposition 2.4]{HN}) Let $X$ be a normal projective variety smooth in codimension $2$ and $D\subset X$ a reduced divisor.
	Suppose $f:X\to X$ is an int-amplified endomorphism such that $f^{-1}(D)=D$ and $f|_{X\backslash D} : X\backslash D
\to X\backslash D$ is quasi-\'etale. Then there is a smooth open subset $U\subseteq X$ such that $D \cap U$ is a normal crossing divisor and $\Codim(X \backslash U) \ge 3$.
	In particular, $\hat{\Omega}_X^1(\log D)$ is locally free over $U$.
\end{proposition}

\begin{proof}
	Let $\nu:\widetilde{D}\to D\subseteq X$ be the normalization of $D$ and $\mathfrak{c}$ the conductor of $D$, regarded as a Weil divisor on $\widetilde{D}$.
	Since $X$ is smooth in codimension $2$, the adjunction formula gives
	$$K_{\widetilde{D}}+\mathfrak{c}=\nu^*(K_X+D)$$
	where $\nu^*(K_X+D)$ is regarded as the pullback of a divisorial sheaf.
	There is an endomorphism $h:\widetilde{D}\to \widetilde{D}$ such that $\nu\circ h=f\circ \nu$ and its ramification divisor $R_h$ is $h^*\mathfrak{c}-\mathfrak{c}$. In fact, we have $K_{\widetilde{D}}+\mathfrak{c}=h^*(K_{\widetilde{D}}+\mathfrak{c})$ from $K_X+D=f^*(K_X+D)$.
	
Note that $h$ is int-amplified and $\mathfrak{c}$ is reduced (cf.~\cite[Theorem 3.3]{Meng_IAMP}, \cite[Lemma 5.3, the arxiv version]{Zh-comp}).
If a plane curve has a reduced conductor over a singular point, then the singularity is nodal.
So $D$ has only normal crossing singularities in codimension one.
\end{proof}

We now apply \cite[Lemma 3.8]{Meng_IAMP} to show the vanishing of the Chern classes.

\begin{proposition}\label{prop-omega-vanishing} Let $X$ be a normal projective variety
	which is of dimension $n \ge 2$ and smooth in codimension $2$, and $D\subset X$ a reduced divisor.
	Suppose $f:X\to X$ is an int-amplified endomorphism such that $f^{-1}(D)=D$ and $f|_{X\backslash D} : X\backslash D
\to X\backslash D$ is quasi-\'etale.
	Let $H$ be an ample divisor on $X$. Then $$c_1(\hat{\Omega}_X^1(\log D))\cdot H^{n-1}= c_1(\hat{\Omega}_X^1(\log D))^2\cdot H^{n-2}= c_2(\hat{\Omega}_X^1(\log D))\cdot H^{n-2}= 0 .$$
\end{proposition}

\begin{proof}
Let the open set $U$ be as in Proposition \ref{codim3-prop}.
Then $f|_{f^{-1}(U)\backslash D}$ is \'etale, since $f|_{X\backslash D}$ is quasi-\'etale
and by the purity of branch loci.
	
	There is a natural morphism $\varphi:f^*\hat{\Omega}_X^1(\log D)\to \hat{\Omega}_X^1(\log D)$ and $\varphi|_{f^{-1}(U)}$ is an isomorphism.
	So for $1\le i\le 2$, we have $$f^*c_i(\hat{\Omega}_X^1(\log D))=c_i(f^*\hat{\Omega}_X^1(\log D))= c_i(\hat{\Omega}_X^1(\log D)).$$
	Then the projection formula implies
	$$c_i(\hat{\Omega}_X^1(\log D))\cdot H^{n-i}=c_i(\hat{\Omega}_X^1(\log D)\cdot (f^t)^*(H^{n-i})/ (\deg f)^t$$
	for any $t>0$.
	By \cite[Lemma 3.8]{Meng_IAMP}, $c_i(\hat{\Omega}_X^1(\log D))\cdot H^{n-i}=0$.
	The proof for $c_1(\hat{\Omega}_X^1(\log D))^2\cdot H^{n-2}=0$, is similar.
\end{proof}

\begin{lemma}\label{lem-H=A+B}
Let $f:X\to X$ be a surjective endomorphism of a projective variety $X$.
Suppose $f^*|_{\N^1(X)}$ is diagonalizable with positive integral eigenvalues $q\ge p$, and no other eigenvalues.
Let $H$ be an ample Cartier divisor.
Then $H=A+B$ for some nef $\Q$-Cartier divisors $A$ and $B$ such that $f^*A\equiv pA$ and $f^*B\equiv qB$.
\end{lemma}

\begin{proof}
	If $p=q$, then $f^*|_{\N^1(X)}= q \id$ and we may take $A=H$ and $B=0$.
	Assume $q>p$.
	Let $\varphi:=f^*|_{\N^1(X)}$.
	Let $A=\lim\limits_{i\to +\infty} p^i\varphi^{-i}(H)$
	and  $B=\lim\limits_{i\to +\infty} \varphi^{i}(H)/q^i$.
	Since $\varphi$ is diagonalizable with only integral eigenvalues $p$ and $q$,
	the above limits are $\Q$-Cartier and $H=A+B$.
	It is clear that $\varphi(A)=pA$ and $\varphi(B)=qB$.
	Note that $A$ and $B$ are limits of ample divisors.
	So $A$ and $B$ are nef.
\end{proof}

We are not able to show the slope semistability for the general int-amplified case. However, the following case is enough for us to rule out Case TIR$_3$ in the proof of Theorem \ref{main-thm-src3}.

\begin{proposition}\label{prop-omega-ss} Let $X$ be a normal projective variety of dimension $n \ge 2$, and $D\subset X$ a reduced divisor.
Suppose $f:X\to X$ is an int-amplified endomorphism such that $f^{-1}(D)=D$ and $f|_{X\backslash D} : X\backslash D
\to X\backslash D$ is quasi-\'etale.
Suppose further that $f^*|_{\N^1(X)}$ is diagonalizable with one or two positive integral eigenvalues and no other eigenvalues.
Let $H$ be an ample divisor on $X$. Then $\hat{\Omega}_X^1(\log D)$ is $H$-slope semistable.
\end{proposition}

\begin{proof}
By Lemma \ref{lem-H=A+B}, we can write $H=A+B$ where $A$ and $B$ are nef $\Q$-Cartier divisors such that $f^*A\equiv pA$ and $f^*B\equiv qB$.
We may assume $q\ge p>1$
(cf.~\cite[Theorem 1.1]{Meng_IAMP}).
Let $\mathcal{F}\subset \hat{\Omega}_X^1(\log D)$ be the maximal destabilizing subsheaf with respect to $H$.
Then:
$$\mu_H(\mathcal{F})=\frac{c_1(\mathcal{F})\cdot H^{n-1}}{\rank \mathcal{F}}
=\sum\limits_{i=0}^{n-1}
{n-1 \choose i}
\frac{c_1(\mathcal{F})\cdot A^i\cdot B^{n-1-i}}{\rank \mathcal{F}}
=\sum\limits_{i=0}^{n-1}
{n-1 \choose i}
\mu_{A^i\cdot B^{n-1-i}}(\mathcal{F}).$$

Suppose the contrary that $\mu_H(\mathcal{F}) > \mu_H(\hat{\Omega}_X^1(\log D)) = 0$
(cf.~Proposition \ref{prop-omega-vanishing}).
Then $\mu_{A^i\cdot B^{n-1-i}}(\mathcal{F})>0$ for some $i$.
In particular, $A^i\cdot B^{n-1-i}\not\equiv_w 0$.
Since $f^*|_{\N^1(X)}$ is diagonalizable, $A^i\cdot B^{n-1-i}\cdot C\neq 0$ for some Cartier divisor $C$ with $f^*C\equiv aC$. Here $a = p$, or $q$, so $a>1$.
By the projection formula, we have
$$(\deg f) A^i\cdot B^{n-1-i}\cdot C=(f^*A)^i\cdot (f^*B)^{n-1-i}\cdot f^*C=(p^iq^{n-1-i}a)  A^i\cdot B^{n-1-i}\cdot C.$$
Therefore, we have $\deg f/p^iq^{n-1-i}=a>1$.
Since $A$ and $B$ are nef, we have
$$s=\sup\{\mu_{A^i\cdot B^{n-1-i}}(\mathcal{F})\,|\, \mathcal{F}\subset \hat{\Omega}_X^1(\log D)\}<\infty.$$
Then for some $k\gg 1$ and $g:=f^k$, we have $$\mu_{A^i\cdot B^{n-1-i}}(g^*\mathcal{F})=(\deg f/p^iq^{n-1-i})^k\mu_{A^i\cdot B^{n-1-i}}(\mathcal{F})
= a^k \mu_{A^i\cdot B^{n-1-i}}(\mathcal{F}) >s.$$
Let the open set $U$ be as in Proposition \ref{codim3-prop}.
Let $j:g^{-1}(U)\hookrightarrow X$ be the inclusion map and let $\mathcal{G}:=j_*((g^*\mathcal{F})|_{g^{-1}(U)})$.
Then $\mu_{A^i\cdot B^{n-1-i}}(\mathcal{G})=\mu_{A^i\cdot B^{n-1-i}}(g^*\mathcal{F})>s$.
Note that $(g^*\mathcal{F})|_{g^{-1}(U)}$ is a subsheaf of the locally free sheaf $(g^*\hat{\Omega}_X^1(\log D))|_{g^{-1}(U)}\cong \hat{\Omega}_X^1(\log D)|_{g^{-1}(U)}$.
Since $\Codim(X \backslash g^{-1}(U)) \ge 2$ and $j_*$ is left exact,
$\mathcal{G}$ is a coherent subsheaf of $\hat{\Omega}_X^1(\log D)$ .
So we get a contradiction.
\end{proof}

With the preparation done, we have the following criterion of toric pairs.

\begin{theorem}\label{thm-iamp-toric} Let $X$ be a rationally connected smooth projective variety and $D\subset X$ a reduced divisor.
	Suppose $f:X\to X$ is an int-amplified endomorphism such that
$f^{-1}(D)=D$ and $f|_{X\backslash D} : X\backslash D \to X\backslash D$ is quasi-\'etale.
	Suppose further $f^*|_{\N^1(X)}$ is diagonalizable with one or two positive integral eigenvalues, and no other eigenvalues.	
Then $(X,D)$ is a toric pair.
\end{theorem}

\begin{proof}
By the assumption, $K_X + D = f^*(K_X + D)$; it is zero in $\N^1(X)$ since $f$ is int-amplified
and hence all eigenvalues of $f^*|_{\N^1(X)}$ are of modulus $> 1$ (cf.~\cite[Theorem 1.1]{Meng_IAMP}).
So $K_X + D \sim 0$, because $X$ is simply connected, and hence has no non-trivial
torsion line bundle (cf.~\cite[Corollary 4.18]{Deb}).
This relation also implies $(X, D)$ is a toric pair when $\dim(X) = 1$.

Assume now $\dim(X) \ge 2$.
By Propositions \ref{prop-omega-vanishing}, \ref{prop-omega-ss} and \cite[Theorem 1.20]{GKP}, $\hat{\Omega}_X^1(\log D)$ is free.
In particular, $h^0(X, \hat{\Omega}_X^1(\log D)=\dim(X)$.
Note that $h^1(X,\mathcal{O}_X)=0$.
By \cite[Theorem 4.5]{MZ_toric}, the complexity $c(X, D)\le 0$ and hence $(X, D)$ is a toric pair by \cite[Theorem 1.2]{BMSZ} (cf.~\cite[Theorem 4.3 and Remark 4.4]{MZ_toric}).
\end{proof}

\begin{proposition}\label{prop-toric-iamp}
	Let $f:X\to X$ be an int-amplified endomorphism of a rationally connected smooth projective variety $X$ with totally invariant ramification, i.e., $f^{-1} (\Supp R_f) = \Supp R_f$.
	Suppose $X$ admits some MMP
	$$X=X_1\dashrightarrow\cdots \dashrightarrow X_r\to Y=\mathbb{P}^1$$
	where $X_i\dashrightarrow X_{i+1}$ is birational and $\pi:X_r\to Y$ is a Fano contraction.
	Then we have:
	\begin{itemize}
	\item[(1)] Replacing $f$ by a positive power, $f^*|_{\N^1(X)}$ is diagonalizable with one or two positive integral eigenvalues, and no other eigenvalues; $f$ descends to int-amplified endomorphism $f_i$ of $X_i$ ($i \le r$), and
each $f_i$ still has totally invariant ramification.
	\item[(2)] $(X_i, \Supp R_{f_i})$ is a toric pair for each $i \le r$.
	\item[(3)] KSC holds for any surjective endomorphism of $X_i$.
	\end{itemize}
\end{proposition}

\begin{proof}
	By \cite[Theorems 1.10 and 1.11]{Meng_IAMP}, replacing $f$ by a positive power, this MMP is $f$-equivariant, $f^*|_{\N^1(X)}$ is diagonalizable with all the eigenvalues being integers greater than $1$, and
all $f_i:=f|_{X_i}$ and $g:=f|_Y$ are still int-amplified.
	Let $\tau:X\dashrightarrow X_r$ be the composition.
	
	Let $W$ be the graph of $\tau$ and let $p_1:W\to X$ and $p_2:W\to X_r$ be the two projections.
	Then $f$ lifts equivariantly to a surjective endomorphism $h:W\to W$.
	Let $E$ be an exceptional prime divisor of $\tau$.
	Write $f^*E=aE$ for some $a>0$.
	Then $h^*E_W=aE_W$ where $E_W$ is the strict transform of $E$ in $W$.
	
If $\pi\circ p_2(E_W)$ is a closed point $y$ of $Y$,
then $E_W$ is contained in the support of $W_y:=p_2^*\pi^*(y)$.
Since $h^*W_y=\delta_g W_y$, we have $a=\delta_g$.
	
Suppose $\pi\circ p_2(E_W)=Y$.
	Since $g$ is polarized, the set $\Per(g)$ of periodic points is Zariski dense in $Y$ by \cite[Theorem 5.1]{Fak}.
	Then $h(F_W)=F_W$ for some (irreducible) general fiber $F_W$ of $\pi\circ p_2$, after replacing $f$ (and $h$) by positive powers.
	Denote by $F:=p_1(F_W)$ and $F_r:=p_2(F_W)$.
	Clearly, $p_1|_{F_W}$ and $p_2|_{F_W}$ are birational morphisms and $F_r$ is also a general fibre of $\pi$.
	Since $E$ dominates $Y$, we have $F\cap E\neq \emptyset$ and hence $E|_F$ is an effective $\Q$-Cartier divisor which is not numerically trivial.
	Note that $f_r^*H_r\equiv q H_r$ for some $\pi$-ample Cartier divisor $H_r$ and integer $q>0$.
	Then $H_r|_{F_r}$ is ample and $(f_r|_{F_r})^*H_r|_{F_r}\equiv qH_r|_{F_r}$.
	Since $f_r$ is int-amplified, $q>1$ (cf.~\cite[Lemma 3.5, Theorem 1.1]{Meng_IAMP}).
	So $f_r|_{F_r}$ is $q$-polarized and hence so is $f|_F$; see \cite[Proposition 1.1 and Corollary 3.12]{MZ}.
	Since $(f|_F)^*E|_F=aE|_F$ and $E|_F\not\equiv 0$, we have $a=q$ (cf.~\cite[Lemma 2.4]{Zh-comp} or \cite[Proposition 2.9]{MZ}).
	
Thus $f^*|_{\N^1(X)}$ has positive integral eigenvalues $\delta_g$ and $q$, and no other eigenvalues. (1) is proved.
Indeed, $R_{f_i}$ is the (birational image) of $R_f$ and $f_i^{-1}(\Supp R_{f_i}) = \Supp R_{f_i}$
holds for $i = 1$ and hence for all $i$.
	
	By (1) and Theorem \ref{thm-iamp-toric}, $(X_i, \Supp R_{f_i})$ is a toric pair
for $i = 1$, and hence for all $i \le r$. Indeed, let $T$ be the big torus acting on $X$.
Then the MMP is $T$-equivariant, and $T$ stabilizes $\Supp R_{f_i}$ for $i = 1$ and hence for all $i$.
(2) is proved.
	
	Since a toric variety is of Fano type, (3) follows from (2) and \cite[Corollary 4.2]{Mat}.
\end{proof}

\begin{proof}[Proof of Theorem \ref{main-thm-src3}]
By \cite[Theorem 1.4]{MZ_PG}, we have the following finite sequence of $G$-equivariant MMP
for some submonoid $G\le \SEnd(X)$ of finite index
$$X=X_1\dashrightarrow\cdots \dashrightarrow X_i\dashrightarrow X_{i+1}\dashrightarrow\cdots X_r\to Y\dashrightarrow \cdots $$
where $X_i\dashrightarrow X_{i+1}$ is birational and $\pi:X_r\to Y$ is a Fano contraction.
Moreover,
$G^*|_{\NS_{\Q}(X)}$ is commutative and $\Q$-diagonalizable.
Let $f$ be a surjective endomorphism of $X$.
Replacing $f$ by a positive power, we may assume $f\in G$.
By Proposition \ref{lem-fano} (cf.~Theorem \ref{main-thm-tir} and its proof), it suffices to show that $f_r:=f|_{X_r}:X_r\to X_r$ does not satisfy Case TIR$_3$.

Suppose the contrary.
Then $\dim(Y)=1$, and $Y\cong \mathbb{P}^1$ since $X$ is rationally connected.
By the assumption, $G$ contains (a positive power of) an int-amplified endomorphism $\mathcal{I}:X\to X$.
Note that $\delta_f=\delta_{f_r}>\delta_{f|_Y}$ and $G^*|_{\NS_{\Q}(X)}$ is commutative and diagonalizable.
Note also that all the eigenvalues of $\mathcal{I}^*|_{\NS_{\Q}(X)}$ are greater than $1$.
Then it is possible for us to take $k\gg 1$ such that $\delta_{f^k \circ \mathcal{I}}\ge (\delta_f)^k>(\delta_{f|_Y})^k\cdot \delta_{\mathcal{I}|_Y}\ge \delta_{(f^k \circ \mathcal{I})|_Y}$.
Moreover, all the eigenvalues of $(f^k \circ \mathcal{I})^*|_{\NS_{\Q}(X)}$ are greater than $1$ (cf.~\cite[Theorem 1.4(2)]{MZ_PG}).
In particular, replacing $f$ by $f^k \circ \mathcal{I}$, we may assume $f$ is also int-amplified and $f_r$ still satisfies Case TIR$_3$ (here $k\gg 1$ is used to make sure $\delta_{f_r}>\delta_{f_r|_Y}$ still holds after replacement).
By Theorem \ref{thm-kappa}, $f_r$ and hence $f$ have totally invariant ramification (the MMP being $G$-equivariant).
By Proposition \ref{prop-toric-iamp}, $X_r$ is toric, contradicting the assumption $\kappa(X_r,-K_{X_r})=0$.
\end{proof}

\end{document}